\documentclass[fleqn,final,3p,times]{elsarticle}
\usepackage{subfigure}
\usepackage{graphicx}
\usepackage[T1]{fontenc}
\usepackage{amssymb}
\usepackage{amsthm}

\usepackage{diagbox}

\usepackage[justification=centering]{caption}

\usepackage[ruled,vlined,linesnumbered]{algorithm2e}

\usepackage{amsmath,mathrsfs}

\usepackage{color}
\usepackage{pgf}
\usepackage{tikz}
\usetikzlibrary{arrows,automata,positioning}
\usetikzlibrary{trees,shapes,calc}
\usepackage[latin1]{inputenc}

\tikzset{
dot/.style = {circle, fill, minimum size=#1,
              inner sep=0pt, outer sep=0pt},
dot/.default = 1pt 
}

\theoremstyle{plain}

\newtheorem{Def}{Definition}[section]
\newtheorem{Eg}{Example}[section]
\newtheorem{Prop}{Proposition}[section]
\newtheorem{Thm}{Theorem}[section]
\newtheorem{Lem}{Lemma}[section]
\newtheorem{Coro}{Corollary}[section]
\newtheorem{Rem}{Remark}[section]

\newproof{Pot2}{Proof of Theorem \ref{Thomo}}

\newcommand{\la}{\langle}
\newcommand{\ra}{\rangle}

\newcommand{\lra}{\longrightarrow}
\newcommand{\rra}{\rightarrow}

\newcommand{\orra}{\overrightarrow}
\newcommand{\olla}{\overleftarrow}
\newcommand{\os}{\overset}

\newcommand{\RNum}[1]{\uppercase\expandafter{\romannumeral #1\relax}}

\biboptions{sort&compress}

\makeatletter
  \newcommand\figcaption{\def\@captype{figure}\caption}
  \newcommand\tabcaption{\def\@captype{table}\caption}
\makeatother

\allowdisplaybreaks[4]
\begin{document}

\begin{frontmatter}

\title{Limited bisimulations for nondeterministic fuzzy transition systems}



\author[Qs]{Sha Qiao}
\ead{sqiao@sdu.edu.cn}

\cortext[cor1]{Corresponding author}
\author[Qs]{Jun-e Feng\corref{cor1}}
\ead{fengjune@sdu.edu.cn}

\author[pz1]{Ping Zhu}
\ead{pzhubupt@bupt.edu.cn}


\address[Qs]{School of Mathematics, Shandong University, Jinan 250100, China}
\address[pz1]{School of Science, Beijing University of Posts and Telecommunications, Beijing 100876,
China}



\begin{abstract}
The limited version of bisimulation, called limited approximate bisimulation, has recently been introduced to fuzzy transition systems (NFTSs).
This article extends limited approximate bisimulation to NFTSs, which are more general structures than FTSs, to introduce a notion of $k$-limited $\alpha$-bisimulation by using an approach of relational lifting, where $k$ is a natural number and $\alpha\in[0,1]$. To give the algorithmic characterization, a fixed point characterization of $k$-limited $\alpha$-bisimilarity is first provided. Then
$k$-limited $\alpha$-bisimulation vector with $i$-th element being a $(k-i+1)$-limited $\alpha$-bisimulation is introduced to investigate conditions for two states to be $k$-limited $\alpha$-bisimilar, where $1\leq i\leq k+1$. Using these results, an $O(2k^2|V|^6\cdot\left|\lra\right|^2)$ algorithm is designed for computing the degree of similarity between two states, where $|V|$ is the number of states of the NFTS and $\left|\lra\right|$ is the greatest number of transitions from states. Finally, the relationship between $k$-limited $\alpha$-bisimilar and $\alpha$-bisimulation under $\widetilde{S}$ is showed, and by which, a logical characterization of $k$-limited $\alpha$-bisimilarity is provided.


\end{abstract}
\end{frontmatter}

\section{Introduction}\label{introd}
Bisimulation \cite{Milner1980bisimulation,Park1981bisimulation} is a useful notion to test the behavioral equivalences for systems or states in a system. It has received growing attention owing to its widely applications in many areas of computer science \cite{Milnerlimited,Park1981bisimulation,Sangiorgi(2009)}, such as model checking by minimizing structures.

Broadly speaking, there are two types of bisimulations, namely, {\it crisp} and {\it fuzzy}, for graph-based structures and fuzzy graph-based structures, where graph-based structures can be labeled transition systems, automata, Boolean networks, social networks and interpretations in description logics, as well as fuzzy graph-based structures can be FTSs, fuzzy automata, fuzzy social networks.
Specifically, crisp bismulations, which are defined as crisp relations, include classical bisimulation, approximate bisimulation, and limited bisimulation.
For instance, to handle the coarseness of fuzzy-language equivalence for FTSs, Cao et al. \cite{Cao(2011)} introduced bisimulations.
Moreover, Li et al. \cite{LiR(2022)} studied bisimulations of probabilistic Boolean networks and minimized the model.
Yang and Li \cite{Yang(2018),Yang1(2020)} introduced approximate bisimulations and minimized the fuzzy automata.
Recently, 
Qiao et al. \cite{Qiao(2022)} investigated approximate bisimulations for NFTSs.

Fuzzy bisimulations include classical fuzzy bisimulation, limited fuzzy bisimulation, approximate bisimulation, and distribution-based fuzzy bisimulation. For example, 
\'{C}iri\'{c} et al. \cite{Ciric(2012),Ciric and J(2012)} introduced fuzzy simulations and fuzzy bisimulations for fuzzy automata. They characterized fuzzy bisimulations by factor fuzzy automata, and illustrated, under fuzzy bisimulations, the invariance of languages. 
Based on the fuzzy bisimulations, Stanimirovi\'{c} et al. \cite{Stanimirovic(2022)} generalized the work \cite{Yang1(2020)} to fuzzy environments. The fuzzy bisimulation proposed perform better than crisp bisimulation \cite{Yang1(2020)} in reduction.
Mici\'{c} et al. \cite{MicicandNguyen(2022)}  
computed the greatest $\lambda$-approximate bisimulation and investigated its existence for fuzzy automata. 
To get better reductions of fuzzy social systems than regular and structural fuzzy relations, they gave the approximate versions of regular and structural fuzzy relations\cite{Micic(2022)}.

In fact, most real-world systems are full of various uncertainties, and so the approximation idea is a important research approach. For instance, Lyu et al. \cite{LyuAppro(2020)} investigated a universal approximation of multi-input multi output fuzzy systems, and Sun and Li \cite{Sun(2023)} discussed the universal approximation of multi-input single-output hierarchical fuzzy system as well as the corresponding algorithmic characterization.


Fan \cite{Fan(2015)} studied fuzzy bisimulation for G\"{o}del modal logic.
Subsequently, Nguyen et al. \cite{Nguyena(2020),NguyenaandTran(2020)} studied bisimulation and bisimilarity for fuzzy description logics, and fuzzy bisimulations for fuzzy structures under the G\"{o}del semantics, where fuzzy bisimulation have potential applications in terms of analyzing fuzzy or weighted social networks. Recently, for fuzzy graph-based structures, a novel method is proposed by \cite{Nguyenand I(2023)} to compute the greatest fuzzy bisimulations and \cite{Nguyen(2023)} computes the greatest fuzzy auto-bisimulation with the fuzzy partition they proposed. 
For NFTSs, to enrich the works \cite{Ciric(2012),Ciric and J(2012)} which are invalid in compositional reasoning, Cao et al. \cite{Cao2013} proposed behavioral distance, and subsequently, Wu et al. \cite{Wu(2016),Wuandchen(2018),WuandDeng(2018)} gave algorithmic and logical characterizations of bisimulations for NFTSs. 
Lately, Qiao et al. \cite{Qiao(2022)} put forward a novel method of lifting a relation via lifting function $\widetilde{S}$, and then introduced $\alpha$-bisimulation under $\widetilde{S}$, where $\alpha\in[0,1]$.

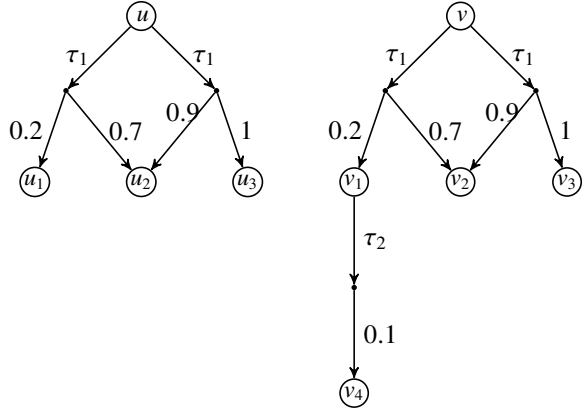
\begin{figure}[h] 
	\tikzstyle{point}=[coordinate,on grid]
	\tikzstyle{solid node}=[circle,draw,inner sep=0.5,fill=black]
	\tikzstyle{state}=[circle,draw,inner sep=0.5,scale=1,font=\bfseries\small]
	\centering
	\begin{tikzpicture}[->,>=stealth',shorten >=0pt,auto,node distance=1.4cm,semithick]
		\node[state]         (D2)  {$u_2$};
		\node[below of =D2] (o2) { };
		\node[state]         (C2) [left of=D2] {$u_1$};
		\node[state]         (E2) [right of=D2] {$u_3$};
		\node[state,node distance=2.2cm]         (B0) [above of=D2] {$\;u\,$};
		\node[solid node]         (C1) [below left of=B0] {};
		\node[solid node]         (D1) [below right of=B0] {};
		
		\path (B0) edge node [left]{$\tau_1$} (C1);
		\path (B0) edge node [right]{$\tau_1$} (D1);
		\path (C1) edge node [left] {0.2} (C2);
		\path (C1) edge node [right] {0.7} (D2);
		\path (D1) edge node [above] {0.9} (D2);
		\path (D1) edge node [right] {1} (E2);
		
		\node[state]         (C2') [right of=E2] {$v_1$};
		\node[state]         (D2') [right of=C2'] {$v_2$};
		\node[state]         (E2') [right of=D2'] {$v_3$};
		
		\node[state,node distance=2.2cm]         (B0') [above of=D2'] {$\;v\,$};
		
		\node[solid node]         (C1') [below left of=B0'] {};
		\node[solid node]         (D1') [below right of=B0'] {};
		
		\node[solid node]         (C3') [below of=C2'] {};
		\node[state]         (D3') [below of=C3'] {$v_4$};
		
		\path (B0') edge node [left]{$\tau_1$} (C1');
		\path (B0') edge node [right]{$\tau_1$} (D1');
		\path (C1') edge node [left] {$0.2$} (C2');
		\path (C1') edge node [right] {$0.7$} (D2');
		\path (D1') edge node [above] {$0.9$} (D2');
		\path (D1') edge node [right] {$1$} (E2');
		\path (C2') edge node {$\tau_2$} (C3');
		\path (C3') edge node {0.1} (D3');
	\end{tikzpicture}
	\caption{States $s$ and $t$ are approximate bisimilar.} \label{Figure1}
\end{figure}

Although $\alpha$-bisimulation under $\widetilde{S}$ \cite{Qiao(2022)} can cope with more cases conveniently than the bisimulations \cite{Cao2013,Wu(2016)},
it is not applicable to some cases. For example, for the states $u$ and $v$ in Fig. \ref{Figure1}, by $\alpha$-bisimulation under $\widetilde{S}$ \cite{Qiao(2022)} (see the following Definition \ref{Def app-bisimulation}), $u$ and $v$ are $0.9$-bisimilar under $\widetilde{S}$ over {\L}ukasiewicz algebra. This is unjustifiable since $u$ can perform action $\tau_2$ to some state after performing action $\tau_1$, while $u$ can not reach any state by the above way. In addition, $u$ and $v$ are 0-bisimilar under $\widetilde{S}$ in G\"{o}del algebra and Product algebra, which is also unreasonable since they can match each other completely except the transition $v_1\os{\tau_2|0.1}\lra v_4$, i.e., they precisely match each other when $u$ and $v$ only perform the action $\tau_1$ or the first transitions from them can match each other completely. Since $u$ can perform $\tau_1$ to distributions $\frac{0.2}{u_1}+\frac{0.7}{u_2}$ and $\frac{0.9}{u_2}+\frac{1}{u_3}$, the system provided in Fig. \ref{Figure1} is an NFTS, and so the limited approximate bisimulation for FTSs introduced in \cite{Qiao(2021)} is invalid. These make it urgent to introduce a new bisimulation for NFTSs.

Fortunately, $k$-limited bisimilarity initiated by Milner \cite{Milner1980bisimulation} provides a idea of proposing new bisimulation. It can measure the similarity of states in the neighbouring subgraphs rather than the whole graphs. 
Recently, Qiao and Zhu \cite{Qiao(2021)} generalized FTSs to quantitative fuzzy transitive systems by equipping the set of labels with a relation to make measuring the similarity between two labels possible, for which, they proposed limited approximate bisimilarity (similarity) and gave some algorithms for computing their greatest scenario. 

This article is devoted to define a new bisimilarity, named $k$-limited $\alpha$-bisimilarity, for more general structures NFTSs, where $k$ is a natural number and $\alpha\in[0,1]$. If state $s$ can perform action $\tau$ to distribution $\mu$, we write $u\os{\tau}\lra p$ and call it a fuzzy transition. If we say that state $u$ can reach the transition $u\os{\tau}\lra p$ for convenience, then the new bisimilarity can measure the degree of similarity by considering the first $k$ transitions the two states studied can reach rather than the whole transitions. The reflexivity, symmetric, and transitivity of $k$-limited $\alpha$-bisimilarity are presented. Thanks to Tarski's fixed-point theorem, $k$-limited bisimulation is a post-fixed point of some suitable monotonic function, and vice versa.
The equivalent condition of $k$-limited $\alpha$-bisimilarity and $\alpha$-bisimilarity under $\widetilde{S}$ is presented, and further some equivalent conditions for two states to be $k$-limited $\alpha$-bisimilar are provided. Resorting to this, the logical characterization of limited approximate bisimulation can be converted to that of approximate bisimulation under lifting fuction $\overline{S}$ in terms of the fuzzy modal logic.
One main contribution is that a polynomial time algorithm is devised to calculate the greatest $\alpha$ such that the two states are $k$-limited $\alpha$-bisimilar, i.e., the degree of $k$-limited similarity between them, is presented.

The reminder of the work is organized as follows. Section \RNum 2 introduces necessary definitions of fuzzy sets and NFTSs, reviews some properties of some operators. In Section \RNum 3, $k$-limited $\alpha$-bisimilarity for NFTSs is defined, and then some properties are presented. 
In Section \RNum 4, a fixed point characterization of $k$-limited $\alpha$-bisimilarity is investigated. 
We study conditions for two states to be $k$-limited $\alpha$-bisimilar and devise an algorithm for computing the degree of $k$-limited similarity between two states in Section \RNum 5.
In Section \RNum 6, the equivalence between $k$-limited
$\alpha$-bisimilarity and $\alpha$-bisimilarity under lifting function $\widetilde{S}$ is derived, and the logical characterization of $k$-limited $\alpha$-bisimilarity is obtained. Finally, Section \RNum 7 makes a brief conclusion and future work.

\section{Preliminaries}
\subsection{Operators and NFTSs}
In this article, the operators $\vee$ and $\wedge$ are denoted respectively by $s\vee t={\rm max}\{s,t\}$ and $s\wedge t={\rm min}\{s,t\}$ for any $s,t\in [0,1]$, i.e., they represent the maximum operation and the minimum operation on $[0,1]$, respectively. Let $\otimes$ be a binary operation on $[0,1]$. If it is commutative: $s\otimes t=t\otimes s$, is associative: $r\otimes(s\otimes t)=(r\otimes s)\otimes t$, is increasing: $r\otimes t\leq s\otimes t$ if $r\leq s$, and has $1$ as the neutral element, then $\otimes$ is called a {\it $t$-norm}, where $r,s,t\in[0,1]$. Additionally, $\otimes$ is called a left continuous $t$-norm if it also satisfies $r\otimes(\bigvee_{i\in I}r_i)=\bigvee_{i\in I}(r\otimes r_i)$, where $I$ is denoted as any index set and $r_i\in [0,1]$. The {\it residuum} $\rightarrow$ of left continuous $t$-norm $\otimes$ is defined by 
\begin{equation}\label{t-norm and residuum}
	t\rightarrow r=\bigvee\{s\in L\mid s\otimes t\leq r\}.
\end{equation}
Three commonly used left continuous $t$-norms are defined for $s,t\in[0,1]$ below.
\begin{table}[h]
	\setlength{\tabcolsep}{9pt}
	\begin{tabular}{|l|l|l|l|}
		\hline 
		& G\"{o}del & Product & {\L}ukasiewicz\\
		\hline
		$s\otimes t$ & $s\wedge t$  & $s\cdot t$ & $(s+t-1)\vee 0$\\
		\hline
		$s\rra t$ &$1$ if $s\leq t$, and  &1 if $s\leq t$, and &$(1-s+t)\wedge1$\\
		&$t$, otherwise 
		&$t/s$, otherwise  & \\
		\hline
	\end{tabular}
\end{table}

The algebra $\mathcal{L}=(L,\wedge,\vee,\otimes,$ $\rightarrow,0,1)$ is called a {\it residuated lattice} \cite{R.Bel(2002),P.Haj(1998)} with 0 being the least element and 1 the greatest element. In this work, the operator $\otimes$ is arranged as a left continuous $t$-norm, the operator $\rightarrow$ the corresponding residuum, $\mathbb{N}$ the set of all natural numbers.


Let $V$ be a nonempty finite set of objects. We denote by $\mathcal{R}(V)$ the set of all fuzzy subsets of $V$, $\mathcal{P}(V)$ the set of all subsets of $V$. The fuzzy set of $U\times V$ is also called a fuzzy relation between $U$ and $V$.
The inverse of the fuzzy relation $R$, denoted by $R^{-1}$, is denoted by $R^{-1}(u,v)=R(v,u)$ for any $(u,v)\in U\times V$.
The {\it support} of a fuzzy set $p$ is a crisp set defined as ${\rm su}(p)=\{u\in V\mid p(u)>0\}$. For more details, we refer the reader to \cite{Cao2013,Ciric(2012)}.

Let $\Sigma$ be a set of labels and $\delta\subseteq S\times \Sigma\times\mathcal{R}(V)$ be a transition relation. Then the triple $\mathcal{S}=(V,\Sigma,\delta)$ is a {\it nondeterministic fuzzy transition system}. 
Obviously, for any $u\in V$ and $\tau\in\Sigma$, $\delta(u,\tau)\subseteq\mathcal{R}(V)$ is a set of fuzzy subsets of $V$. If $p\in\delta(v,\tau)$, we write $u\os{\tau}\lra p$ and call it a {\it fuzzy transition}. Furthermore, we also can say that $u$ can reach the distribution $p$ after 1 transition with label $\tau$ or $u$ can reach the transition $u\os{\tau}\lra p$. If we define a {\it label sequence} as the sequence with elements belonging to $\Sigma$, then in Fig. \ref{Figure1}, $v$ can reach the distribution $\frac{0.1}{v_3}$ after 2 transitions with the label sequence $(\tau_1,\tau_2)$. Set $\delta$ is said to be finite if $\delta(u,\tau)$ is finite for any $u\in V$ and $\tau\in\Sigma$. An NFTS is called finite if $V$, $\Sigma$, and $\delta$ are all finite. For later use, we define by $R_{\delta}=\{(u,\tau|\gamma,v)\mid\exists p\in\delta(u,\tau)\mbox{ such that } p(u)=\gamma>0\}$ the set of {\it labeled fuzzy directed edges} of $\mathcal{S}$, which is also a ternary relation. 
The successor neighborhood and predecessor neighborhood of $X\subseteq V$ are given by
\begin{align*}
	(\orra{R_{\delta}})_X=&\{v\in V\mid (u,\tau|\gamma,v)\in R_{\delta}, u\in X\} \quad \mbox{and} \quad (\olla{R_{\delta}})_X=\{v\in V\mid (v,\tau|\gamma,u)\in R_{\delta}, u\in X\},
\end{align*}
respectively. If $\delta\subset V\times V$, then the neighborhoods degenerate to classic neighborhoods \cite{Pawlak(1982)}.

A state $u$ is called a steady state if $u$ can not reach any states in $\mathcal{S}$, i.e., $(\overrightarrow{R_{\delta}})_u=\emptyset$. 
A {\it path} from $u_0$ to $u_n$ is defined as a finite sequence $(u_0,\tau_1|\gamma_1,u_1,\tau_2|\gamma_2,u_2,\cdots,u_{n-1}$, $\tau_{n}|\gamma_n,u_n)$ satisfying that $(u_i,\tau_{i+1}|\gamma_{i+1},u_{i+1})\in R_{\delta}$ for all $0\leq i\leq n-1$. The number of labeled fuzzy directed edges in a path $p$ is defined as the {\it length}, denoted by $len(p)$. Furthermore, the maximum length of paths from $u_0$, recorded as $l(u_0)$, is defined as maximum length of paths from $u$.

From now on, unless specifically noted, $\epsilon\in[0,1]$, and there is no distinction between using fuzzy sets and distributions. For simplicity, if and only if is abbreviated as iff.

\subsection{Lifting of a relation}
A relation $R\subseteq V\times V$ is lifted via the {\it lifting function} $\widetilde{S}$ \cite{Qiao(2022)}:
\begin{align*}
	\widetilde{S}\colon\mathcal{R}(V)\times\mathcal{R}(V)\times\mathcal{P}(V\times V)&\lra [0,1],\\
	(p,q,R)&\longmapsto
	\bigwedge_{u\in V}((p(s)\rra q(\orra{R}_u))\wedge(q(u)\rra p(\olla{R}_u))).
\end{align*}
below.

\begin{Def}[see \cite{Qiao(2022)}]\label{relation lifting}
	{\rm Let $R\subseteq V\times V$ and $\alpha\in[0,1]$. The lifted relation $R_{\alpha}^{\dag}\subseteq\mathcal{R}(V)\times\mathcal{R}(V)$ is a relation over possibility distributions such that $(p,q)\in R_{\alpha}^{\dag}$ iff $\widetilde{S}(p,q,R)\geq\alpha$.
	}
\end{Def}

In fact, $(p,q)\in R_{\alpha}^{\dag}$ when and only when $(p,q)\in(({\rm su}(p)\times {\rm su}(q))\cap R)_{\alpha}^{\dag}$, and so $(p,q)\in(({\rm su}(p)\times {\rm su}(q))\cap R)_{\alpha}^{\dag}$ iff $\widetilde{S}(p,q,R)\geq\alpha$.
The operation $R^{\dag}_{\alpha}$ has the following properties.
\begin{Lem}[see \cite{Qiao(2022)}]\label{Prop lifted relation}
	{ Let $\alpha,\alpha_1,\alpha_2\in[0,1]$.
		\begin{enumerate}
			\item[(1)] $(I_V)^{\dag}_{\alpha}=\{(p,q)\in\mathcal{R}(V)\times\mathcal{R}(V)\mid p(u)\rra q(v)\wedge q(v)\rra p(v)\geq\alpha \mbox{ for any } v\in V\}$, $\mbox{ where }  I_V=\{(u,u)\,|\, u\in V\}$.
			\item[(2)] $R_1\subseteq R_2 \mbox{ and } \alpha_1\leq\alpha_2 $ implies $ ({R_1})^{\dag}_{\alpha_2}\subseteq ({R_2})^{\dag}_{\alpha_1}$.
			\item[(3)] $(R_1)_{\alpha}^{\dag}\cup (R_2)_{\alpha}^{\dag}\subseteq (R_1\cup R_2)_{\alpha}^{\dag}$.
			\item[(4)] $(R_1)^{\dag}_{\alpha_1}\circ(R_2)^{\dag}_{\alpha_2}\subseteq(R_1\circ R_2)^{\dag}_{\alpha_1\otimes\alpha_2}$.
			\item[(5)] $({R^{-1}})^{\dag}_{\alpha}=(R_{\alpha}^{\dag})^{-1}$.
		\end{enumerate}
	}
\end{Lem}

In NFTSs, $\alpha$-bisimulation under $\widetilde{S}$ is recalled below.
\begin{Def}[see \cite{Qiao(2022)}]\label{Def app-bisimulation}
	{\rm Let $\mathcal{S}=(V,\Sigma,\delta)$ be an NFTS and $\alpha\in[0,1]$. A relation $R\subseteq V\times V$ is called an {\it $\alpha$-simulation} under $\widetilde{S}$ if for any $(u,v)\in R$ and for each $\tau\in\Sigma$, $u\os{\tau}\lra p$ implies that $v\os{\tau}\lra q$ such that $(p,q)\in R^{\dag}_{\alpha}$. If both $R$ and $R^{-1}$ are $\alpha$-simulations under $\widetilde{S}$, then $R$ is called an {\it $\alpha$-bisimulation} under $\widetilde{S}$.
	}
\end{Def}

Like \cite{Qiao(2022)}, two states $u$ and $v$ are $\alpha$-bisimilar under $\widetilde{S}$, described as $u\sim^{\alpha}v$, if $(u,v)$ belongs to some $\alpha$-bisimulation under $\widetilde{S}$. In this article, we also call that $\sim^{\alpha}$ an $\alpha$-bisimulation under $\widetilde{S}$ in $\mathcal{S}$, and at the same time, if two states $u$ and $v$ are $\alpha$-bisimilar under $\widetilde{S}$, then we also call $u$ and $v$ are $\alpha$-bisimilar under $\widetilde{S}$ in $\mathcal{S}$.

\section{The $k$-limited $\alpha$-bisimilarity}
Motivated by \cite{Milnerlimited,Qiao(2021),Qiao(2022)}, we define a $k$-limited $\alpha$-bisimilarity.
Comparing with $\alpha$-bisimulation under $\widetilde{S}$, $k$-limited $\alpha$-bisimilarity only requires any possibility distribution one state can reach after no more than $k$ transitions with some label sequence can be matched by a distribution the other related state can reach after the same number of transitions with the same label sequence, and vice versa.

\begin{Def}\label{Def k-limited bisimulation}
	{\rm Let $\mathcal{S}=(V,\Sigma,\delta)$ be an NFTS, $\alpha\in[0,1]$, and $k\in\mathbb{N}$. The relation {\it $k$-limited $\alpha$-bisimilarity} $\approx_k^{\alpha}$ over $V$ is defined as follows.
		\begin{enumerate}
			\item[(1)] $u\approx_0^{\alpha} v$, for any $u,v\in V$.
			\item[(2)] $u\approx_{i+1}^{\alpha} v$, $i\geq 0$, if for each $\tau\in\Sigma$, the following hold\\
			i)  $u\os{\tau}\lra p$ implies $v\os{\tau}\lra q$ such that $(p,q)\in(\approx_i^{\alpha})_{\alpha}^{\dag}$;\\
			ii) $v\os{\tau}\lra q$ implies $u\os{\tau}\lra p$ such that $(p,q)\in(\approx_i^{\alpha})_{\alpha}^{\dag}$.
		\end{enumerate}
	}
\end{Def}

Two states $u,v\in V$ are {\it $k$-limited $\alpha$-bisimilar} if $u\approx_k^{\alpha}v$. The degree of $k$-limited similarity between $u$ and $v$ is defined as the greatest $\alpha$ such that $u\approx_k^{\alpha}v$. In fact, $k$-limited $\alpha$-bisimilarity is the greatest $k$-limited $\alpha$-bisimulation contained in $V\times V$.
Based on $k$-limited $\alpha$-bisimilarity, $k$-limited $\alpha$-bisimulation is defined as a subset of $k$-limited $\alpha$-bisimilarity. Let $\mathcal{S}_1=(V_1,\Sigma_1,\delta_1)$ and $\mathcal{S}_2=(V_2,\Sigma_2,\delta_2)$ be two NFTSs. If $R\subseteq V_1\times V_2$ is a $k$-limited $\alpha$-bisimulation, then $R$ is called a $k$-limited $\alpha$-bisimulation between
$\mathcal{S}_1$ and $\mathcal{S}_2$.

By the following example, we illustrate $k$-limited $\alpha$-bisimilarity.
\begin{Eg}
	{\rm In the NFTS $\mathcal{S}=(V,\Sigma,\delta)$ presented in Fig. \ref{Figure1}, where
		\begin{align*}
			&S=\{u,u_1,u_2,u_3,v,v_1,v_2,v_3,v_4\},\\
			&\delta(u,\tau_1)=\{p_1,p_2\},\  \delta(v,\tau_1)=\{q_1,q_2\},\ \delta(v_1,\tau_2)=\{q_3\},\\
			&p_1=\frac{0.2}{u_1}+\frac{0.7}{u_2}, \quad\quad
			p_2=\frac{0.9}{u_2}+\frac{1}{u_3},\\
			&q_1=\frac{0.2}{v_1}+\frac{0.7}{u_2},
			\quad\quad\, q_2=\frac{0.9}{v_2}+\frac{1}{v_2}, \quad\quad\, q_3=\frac{0.1}{v_4}.
		\end{align*}
		Let $\otimes$ be a Product $t$-norm. By Definition \ref{Def k-limited bisimulation}, $u\approx_1^1v$ and $\{(u,v),(v,u)\}$ is a $1$-limited $1$-bisimulation.
	}
\end{Eg}

$k$-limited $\alpha$-bisimulation has the following several properties.

\begin{Prop}\label{Prop k-limited bisimulation}
	{ Let $\mathcal{S}=(V,\Sigma,\delta)$ be an NFTS, $k,k_i\in\mathbb{N}$ and $\alpha,\alpha_i\in[0,1]$, $i=1,2$. We have the following statements.
		\begin{enumerate}
			\item[(1)] If $k_1\geq k_2$ and $\alpha_1\geq\alpha_2$, then $\approx_{k_1}^{\alpha_1}\subseteq\approx_{k_2}^{\alpha_2}$.
			\item[(2)] $I_V=\{(u,u)\,|\, u\in V\}$ is a $k$-limited $\alpha$-bisimulation for any $k\in\mathbb{N}$ and $\alpha\in[0,1]$.
			\item[(3)] If $R_i$ is a $k_i$-limited $\alpha_i$-bisimulation, $i=1,2$, then $R_1\cap R_2$ is a $(k_1\wedge k_2)$-limited $(\alpha_1\vee\alpha_2)$-bisimulation and is also a $(k_1\vee k_2)$-limited $(\alpha_1\wedge\alpha_2)$-bisimulation, and $R_1\cup R_2$ is a ($k_1\wedge k_2$)-limited ($\alpha_1\wedge\alpha_2$)-bisimulation.
			\item[(4)] If $R$ is a $k$-limited $\alpha$-bisimulation, then so is $R^{-1}$.
			\item[(5)] If $R_i$ is a $k_i$-limited $\alpha_i$-bisimulation, $i=1,2$, then $R_1\circ R_2$ is a $(k_1\wedge k_2)$-limited $(\alpha_1\otimes\alpha_2)$-bisimulation.
		\end{enumerate}
	}
\end{Prop}
\begin{proof}
	By induction on $k_2$, (1) is obvious. (2), (3) and (4) follow directly from (1), Definition \ref{Def k-limited bisimulation} and Lemma \ref{Prop lifted relation} (5), and we thus omit the proof.
	
	(5) We prove it by induction on $k_1\wedge k_2$. By Definition \ref{Def k-limited bisimulation}, the case of $k_1\wedge k_2=0$ holds. Under the condition that the case of $k_1\wedge k_2=m$ also holds, we prove that the case of $k_1\wedge k_2=m+1$ holds, where $m\geq 0$. A $(u_1,u_2)\in R_1\circ R_2$ implies that a state $u_3\in V$ exists such that $(u_1,u_3)\in R_1$ and $(u_3,u_2)\in R_2$. Since $R_1$ is a $k_1$-limited $\alpha_1$-bisimulation, we have that for any $u_1\os{\tau}\lra p_1$, there exists $u_3\os{\tau}\lra p_3$ such that $(p_1,p_3)\in(\approx_{k_1-1}^{\alpha_1})_{\alpha_1}^{\dag}$. Similarly, for $u_3\os{\tau}\lra p_3$, there exists $u_2\os{\tau}\lra p_2$ such that $(p_3,p_2)\in(\approx_{k_2-1}^{\alpha_2})_{\alpha_2}^{\dag}$. Hence $(p_1,p_2)\in(\approx_{k_1-1}^{\alpha_1}\circ \approx_{k_2-1}^{\alpha_2})_{\alpha_1\otimes\alpha_2}^{\dag}$ by Lemma \ref{Prop lifted relation} (4). By the induction hypothesis, $\approx_{k_1-1}^{\alpha_1}\circ \approx_{k_2-1}^{\alpha_2}$ is a ($k_1\wedge k_2-1$)-limited $(\alpha_1\otimes\alpha_2)$-bisimulation. By (3),
	$\approx_{k_1-1}^{\alpha_1}\circ \approx_{k_2-1}^{\alpha_2}\subseteq\approx_{k_1\wedge k_2-1}^{\alpha_1\otimes\alpha_2}$, and so  $(p_1,p_2)\in(\approx_{k_1\wedge k_2-1}^{\alpha_1\otimes\alpha_2})_{\alpha_1\otimes\alpha_2}^{\dag}$ by Lemma \ref{Prop lifted relation} (2).
	Similarly, for any $u_2\os{\tau}\lra p_2$, there exists $u_1\os{\tau}\lra p_1$ such that $(p_1,p_2)\in(\approx_{k_1\wedge k_2-1}^{\alpha_1\otimes\alpha_2})_{\alpha_1\otimes\alpha_2}^{\dag}$. Therefore, by Definition \ref{Def k-limited bisimulation}, $R_1\circ R_2$ is a $(k_1\wedge k_2$)-limited $(\alpha_1\otimes\alpha_2)$-bisimulation.
\end{proof}

Proposition \ref{Prop k-limited bisimulation} (3) implies that $k$-limited $\alpha$-bisimilarity is the union of all $k$-limited $\alpha$-bisimulations.

\section{Fixed-point characterization}\label{Fixed-point char}
Using fixed-point theory (see, for example, \cite{Buchholz(2008),Cao(2011),Sun(2009),Qiao(2022),QiaoandFeng(2023),Petkov2006}, and the therein), in this section, the fact that a relation is a $k$-limited $\alpha$-bisimilarity when and only when it is a fixed point of a suitable monotonic function is provided.
Let $\alpha\in[0,1]$. On the lattice $(V\times V,\subseteq)$, we recursively define the following function:
\begin{align*}
	F_k^{\alpha}\colon\ \mathcal{P}(V\times V)&\lra \mathcal{P}(V\times V), \ k\in\mathbb{N},\\
	F_0^{\alpha}(R_1)&=R_1;\\
	F_{i+1}^{\alpha}(R_1)&=\{(u_{i+1},v_{i+1})\in V\times V\,|
	\forall u_{i+1}\os{\tau}\lra p, \, \exists v_{i+1}\os{\tau}\lra q \mbox{ such that }(p,q)\in(F_i^{\alpha}((\overrightarrow{R}_{\delta})_{(\overleftarrow{R_1})_V}\times(\overrightarrow{R_{\delta}})_{(\overrightarrow{R_1})_V}))_{\alpha}^{\dag},\\
	&\forall v_{i+1}\os{\tau}\lra q,\, \exists u_{i+1}\os{\tau}\lra p \mbox{ such that } (p,q)\in (F_i^{\alpha}((\overrightarrow{R_{\delta}})_{(\overleftarrow{R_1})_V}\times(\overrightarrow{R_{\delta}})_{(\overrightarrow{R_1})_V}))_{\alpha}^{\dag}\},\quad 0\leq i\leq k-1.
\end{align*}

We now show the monotonicity of the function $F_k^{\alpha}$.

\begin{Prop}
	{Let $k\in\mathbb{N}$, $R_i\in\mathcal{P}(V\times V)$ and $\alpha_i\in[0,1]$, $i=1,2$. If $R_1\subseteq R_2$ and $\alpha_1\leq\alpha_2$, then $F_k^{\alpha_2}(R_1)\subseteq F_k^{\alpha_1}(R_2)$.
	}
\end{Prop}
\begin{proof}
	It is easy to get that $F_0^{\alpha_2}(R_1)\subseteq F_0^{\alpha_1}(R_2)$. With the help of $F_m^{\alpha_2}(R_1)\subseteq F_m^{\alpha_1}(R_2)$, $m\geq0$, we prove that $F_{m+1}^{\alpha_2}(R_1)\subseteq F_{m+1}^{\alpha_1}(R_2)$.
	$R_1\subseteq R_2$ implies that
	\begin{align*}
		(\overrightarrow{R}_{\delta})_{(\overleftarrow{R_1})_V}\times(\overrightarrow{R_{\delta}})_{(\overrightarrow{R_1})_V}\subseteq (\overrightarrow{R}_{\delta})_{(\overleftarrow{R_2})_V}\times(\overrightarrow{R_{\delta}})_{(\overrightarrow{R_2})_V}, 
	\end{align*}
	and so by the induction hypothesis,
	\begin{align*}
		&F_m^{\alpha_2}((\overrightarrow{R}_{\delta})_{(\overleftarrow{R_1})_V}\times(\overrightarrow{R_{\delta}})_{(\overrightarrow{R_1})_V})\subseteq F_m^{\alpha_1}((\overrightarrow{R}_{\delta})_{(\overleftarrow{R_2})_V}\times(\overrightarrow{R_{\delta}})_{(\overrightarrow{R_2})_V}), 
	\end{align*}
	By Lemma \ref{Prop lifted relation} (2),
	\begin{align*}
		&(F_m^{\alpha_2}((\overrightarrow{R}_{\delta})_{(\overleftarrow{R_1})_V}\times(\overrightarrow{R_{\delta}})_{(\overrightarrow{R_1})_V}))_{\alpha_2}^{\dag}\subseteq(F_m^{\alpha_1}((\overrightarrow{R}_{\delta})_{(\overleftarrow{R_1})_V}\times(\overrightarrow{R_{\delta}})_{(\overrightarrow{R_1})_V}))_{\alpha_1}^{\dag}, 
	\end{align*}
	and hence $F^{\alpha_2}_{m+1}(R_1)\subseteq F^{\alpha_1}_{m+1}(R_2)$ by the definition of $F_n^{\alpha}$, $n\in\mathbb{N}$. Therefore $F_n^{\alpha}$, $n\in\mathbb{N}$, is a monotonic function.
\end{proof}

The $k$-limited $\alpha$-bisimulation is now proved to be a post-fixed point of $F_k^{\alpha}$.
\begin{Thm}\label{fixed-point characterization}
	{ The following statements hold for $\alpha\in[0,1]$ and $k\in\mathbb{N}$.
		\begin{enumerate}
			\item[(1)] $R$ is a $k$-limited $\alpha$-bisimulation iff $R$ is a post-fixed point of $F_k^{\alpha}$.
			\item[(2)] $R_1$ and $R_2$ are post-fixed points of $F_k^{\alpha}$, then so is $R_1\cup R_2$.
		\end{enumerate}
	}
\end{Thm}
\begin{proof}
	(1) (Sufficiency.) We first show that $F_k^{\alpha}(R)$ is a $k$-limited $\alpha$-bisimulation. Evidently, $F_0^{\alpha}(R)$ is a 0-limited $\alpha$-bisimulation. Suppose that $F_m^{\alpha}(R)$ is an $m$-limited $\alpha$-bisimulation. Then $F_m^{\alpha}(R)\subseteq\approx_m^{\alpha}$, which implies that $(F_m^{\alpha}(R))_{\alpha}^{\dag}\subseteq(\approx_m^{\alpha})_{\alpha}^{\dag}$ by Proposition \ref{Prop lifted relation}. By Definition \ref{Def k-limited bisimulation} and the definition of $F_{m+1}^{\alpha}$, $F_{m+1}^{\alpha}(R)$ is an $(m+1)$-limited $\alpha$-bisimulation. Hence, if $R\subseteq F_{m+1}^{\alpha}(R)$, $R$ is an $(m+1)$-limited $\alpha$-bisimulation. The sufficiency holds.
	
	(Necessity.) It is trivial for $k=0$. Under the condition that the the case of $k=m$ holds, $m\geq0$, and $R$ is an $(m+1)$-limited $\alpha$-bisimulation, we see that for any $(u,v)\in R$, $u\os{\tau}\lra p$ implies $v\os{\tau}\lra q$ that satisfies $(p,q)\in(\approx_{m}^{\alpha})_{\alpha}^{\dag}$. By Definition \ref{relation lifting}, we have $(p,q)\in(({\rm su}(p)\times {\rm su}(q))\cap\approx_{m}^{\alpha})_{\alpha}^{\dag}$. Since $({\rm su}(p)\times {\rm su}(q))\cap\approx_{m}^{\alpha}$ is an $m$-limited $\alpha$-bisimulation, $({\rm su}(p)\times {\rm su}(q))\cap\approx_{m}^{\alpha}\subseteq F_{m}^{\alpha}(({\rm su}(p)\times {\rm su}(q))\cap\approx_{m}^{\alpha})$ by the induction hypothesis. Hence,
	\begin{align*}
		(p,q)\in&(F_{m}^{\alpha}(({\rm su}(p)\times {\rm su}(q))\cap\approx_{m}^{\alpha}))_{\alpha}^{\dag}\subseteq(F_{m}^{\alpha}((\overrightarrow{R}_{\delta})_{(\overleftarrow{R_1})_V}\times(\overrightarrow{R_{\delta}})_{(\overrightarrow{R_1})_V}))_{\alpha}^{\dag}.
	\end{align*}
	Therefore $R\subseteq F_{m+1}^{\alpha}(R)$ by the definition of $F_k^{\alpha}$.
	
	(2) By statement (1), $R_i$ is a $k$-limited $\alpha$-bisimulation, $i=1,2$, then by Proposition \ref{Prop k-limited bisimulation} (3), $R_1\cup R_2$ is a $k$-limited $\alpha$-bisimulation, and so $R_1\cup R_2$ is a fixed point of $F_k^{\alpha}$ by statement (1) again.
\end{proof}

The following fact follows from Theorem \ref{fixed-point characterization} immediately.

\begin{Prop}
	{ Let $\alpha\in[0,1]$. $\approx_k^{\alpha}$ is the greatest fixed-point of $F_k^{\alpha}$, i.e., $\approx_k^{\alpha}=\cup\{R\in\mathcal{P}(V\times V)\,|\,R\subseteq F_k^{\alpha}(R)\}$.
	}
\end{Prop}

\section{Algorithmic characterization}
This section starts with some conditions for two states to be $k$-limited bisimilar, which is the foundation for devising an algorithm for determining the degree of $k$-limited similarity.

\subsection{Conditions for two states to be $k$-limited $\alpha$-bisimilar}
To study the relationship between $k$-limited $\alpha$-bisimulation and $\alpha$-bisimulation under $\widetilde{S}$, $k$-limited $\alpha$-bisimulation vector is proposed, which is composed of $i$-limited $\alpha$-bisimulations, $0\leq i\leq k$.

\begin{Def}\label{Def k-limited bisimulation vector}
	{\rm Let $\mathcal{S}=(V,\Sigma,\delta)$ be an NFTS, $k\in\mathbb{N}$ and $\alpha\in[0,1]$. Let $B$ be a $(k+1)$-vector with $B[i]\subseteq V\times V$, $1\leq i\leq k+1$. Vector $B$ is called a {\it $k$-limited $\alpha$-bisimulation vector} if for any $(u',v')\in B[i]$ and $\tau\in\Sigma$, $1\leq i\leq k$, the following hold:
		\begin{enumerate}
			\item[(1)]\ $u\os{\tau}\lra p$ implies $v\os{\tau}\lra q$ such that $(p,q)\in (B[i+1])_{\alpha}^{\dag}$.
			\item[(2)]\ $v\os{\tau}\lra q$ implies $u\os{\tau}\lra p$ such that $(p,q)\in(B[i+1])_{\alpha}^{\dag}$.
		\end{enumerate}
	}
\end{Def}

The greatest $k$-limited $\alpha$-bisimulation vector contained in a $k$-vector is called the $k$-limited $\alpha$-bisimilarity vector contained in the $k$-vector. By the following example, $k$-limited $\alpha$-bisimulation vector is illustrated.
\begin{Eg}
	{\rm Reconsider the NFTS $\mathcal{S}=(V,\Sigma,\delta)$ presented in Fig. \ref{Figure2}. Let $\mathcal{L}$ be a Product algebra. Then
		$(\{(u,v)\},$ $\{(u_1,v_1),(u_2,v_2)\},\{(u_3,v_3),(u_4,v_4)\})$
		is a $2$-limited $\frac{2}{3}$-bisimulation vector.
	}
\end{Eg}

By Definition \ref{Def k-limited bisimulation vector}, $u\approx_{i}^{\alpha}v$ can be obtained from $(u,v)\in B[k-i+1]$ under the condition that $B$ is a $k$-limited $\alpha$-bisimulation vector.
In addition, any $(k+1)$-vector $B$ can be referred to as a $k$-limited $\alpha$-bisimulation vector for some $\alpha\in[0,1]$, and so the degree of $k$-limited similarity between them can be measured resorting to $(k+1)$-vector.
In order to show this fact conveniently, for any $k$-vector $B$ with elements belonging to $\mathcal{P}(V\times V)$, let
\begin{align*}
	Q(B)=&\bigwedge_{1\leq i\leq k-1}\bigwedge_{(u_i,v_i)\in B[i]}(\bigwedge_{u_i\os{\tau}\lra p}\bigvee_{v_i\os{\tau}\lra q}\widetilde{S}(p,q,B[i+1])\wedge\bigwedge_{v_i\os{\tau}\lra q}\bigvee_{u_i\os{\tau}\lra p}\widetilde{S}(p,q,B[i+1])).
\end{align*}

\begin{Prop}\label{Degree of bismilation vector}
	{Let $\mathcal{S}=(V,\Sigma,\delta)$ be an NFTS, $k\in\mathbb{N}$ and $\alpha\in[0,1]$. For any $k$-vector with elements being a relation on $V$, vector $B$ is a $k$-limited $\alpha$-bisimulation vector iff $\alpha\leq Q(B)$.
	}
\end{Prop}
\begin{proof}
	(Sufficiency.) Assume that $\alpha\leq Q(B)$. By $Q(B)$, we have
	\begin{align*}
		\alpha\leq&\bigwedge_{1\leq i\leq k}\bigwedge_{(u_i,v_i)\in B[i]}(\bigwedge_{u_i\os{\tau}\lra p}\bigvee_{v_i\os{\tau}\lra q}\widetilde{S}(p,q,B[i+1])\wedge\bigwedge_{v_i\os{\tau}\lra q}\bigvee_{u_i\os{\tau}\lra p}\widetilde{S}(p,q,B[i+1])).
	\end{align*}
	We know that $B[k+1]$ is a 0-limited $\alpha$-bisimulation. Assume that $B[i]$ is a ($k-i+1$)-limited $\alpha$-bisimulation, $2\leq i\leq k+1$. We now show that $B[i-1]$ is a ($k-i+2$)-limited $\alpha$-bisimulation. 
	For any $r=(u_{i-1},v_{i-1})\in B[i-1]$, by
	\begin{align*}
		\bigwedge_{u_{i-1}\os{\tau}\lra p}\bigvee_{v_{i-1}\os{\tau}\lra q}\widetilde{S}(p,q,B[i])\geq\alpha,
	\end{align*}
	we have that for any $u_{i-1}\os{\tau}\lra p$, there exists $v_{i-1}\os{\tau}\lra q$ such that $\widetilde{S}(p,q,B[i])\geq\alpha$, i.e.,
	$(p,q)\in(B[i])_{\alpha}^{\dag}$. Similarly, for any $v_{i-1}\os{\tau}\lra q$, there exits $u_{i-1}\os{\tau}\lra p$ such that $(p,q)\in(B[i])_{\alpha}^{\dag}$. By Definition \ref{Def k-limited bisimulation}, $B[i-1]$ is a ($k-i+2$)-limited $\alpha$-bisimulation. Hence $B$ is a $k$-limited $\alpha$-bisimulation vector by Definition \ref{Def k-limited bisimulation vector}. The sufficiency holds.
	
	(Necessity.) By contradiction, assume that  $\alpha>Q(B)$ for some $\alpha\in[0,1]$ such that $B$ is a $k$-limited $\alpha$-bisimulation vector. Then there exists $(u_i,v_i)\in B[i]$ for some $1\leq i\leq k$ such that
	\begin{align*}
		&\bigwedge_{u_i\os{\tau}\lra p}\bigvee_{v_i\os{\tau}\lra q}\widetilde{S}(p,q,B[i+1])\wedge\bigwedge_{v_i\os{\tau}\lra q}\bigvee_{u_i\os{\tau}\lra p}\widetilde{S}(p,q,B[i+1])<\alpha.
	\end{align*}
	Without affecting the results, we can suppose that
	\begin{align*}
		\bigwedge_{u_i\os{\tau}\lra p}\bigvee_{v_i\os{\tau}\lra q}\widetilde{S}(p,q,B[i+1])<\alpha.
	\end{align*}
	This means that there exists $u_i\os{\tau}\lra p$ that satisfies that there is no $v_i\os{\tau}\lra q$ satisfying $(p,q)\in(B[i+1])_{\alpha}^{\dag}$. Hence $B$ is not a $k$-limited $\alpha$-bisimulation vector. This contradicts the hypothesis that $B$ is a $k$-limited $\alpha$-bisimulation vector. Therefore $\alpha\leq Q(B)$.
\end{proof}

We now reveal the equivalence between $k$-limited $\alpha$-bisimulation and $\alpha$-bisimulation under $\widetilde{S}$. 
\begin{Prop}
	{Given an NFTS $\mathcal{S}=(V,\Sigma,\delta)$, $u,v\in V$, $k\in\mathbb{N}$ and $\alpha\in[0,1]$. Let $B$ be a $(k+1)$-vector with $B[i+1]\subseteq D_u[i]\times D_v[i]$, $0\leq i\leq k$. Then $B$ is a $k$-limited $\alpha$-bisimulation vector iff $\cup_{1\leq i\leq k+1}B[i]$ is an $\alpha$-bisimulation under $\widetilde{S}$ between $T_{\mathcal{S}}(u,k)$ and $T_{\mathcal{S}}(v,k)$.
	}
\end{Prop}
\begin{proof}
	(Sufficiency.) Assume that $\cup_{1\leq i\leq k+1}B[i]$ is an $\alpha$-bisimulation under $\widetilde{S}$ between $T_{\mathcal{S}}(u,k)$ and $T_{\mathcal{S}}(v,k)$. It is obvious that $B[k+1]$ is a 0-limited $\alpha$-bisimulation. Suppose that $B[g+1]$ is a $(k-g)$-limited $\alpha$-bisimulation, $1\leq g\leq k$. Then we prove that $B[g]$ is a $(k-g+1)$-limited $\alpha$-bisimulation. For any $(u_{g-1},v_{g-1})\in B[g]$, since $u_{g-1}\sim^{\alpha}v_{g-1}$ and the equations in (\ref{set of k transitions}), $u_{g-1}\os{\tau}\lra p$ implies $v_{g-1}\os{\tau}\lra q$ that satisfies
	\begin{align*}
		&(p,q)\in(({\rm su}(p)\times {\rm su}(q))\cap(\cup_{1\leq i\leq k+1}B[i]))_{\alpha}^{\dag}\subseteq(B[g+1])_{\alpha}^{\dag}.
	\end{align*}
	By the hypothesis that $B[g+1]$ is a $(k-g)$-limited $\alpha$-bisimulation, we conclude that $B[g+1]$ is a $(k-g)$-limited $\alpha$-bisimulation, and so
	$(\mu,\eta)\in(B[g+1])_{\alpha}^{\dag}\subseteq(\approx_{k-g}^{\alpha})_{\alpha}^{\dag}$.
	The symmetric case that any transition from $v_{g-1}$ can be analyzed analogously, and we thus omit it. By Definition \ref{Def k-limited bisimulation}, $u_{g-1}\approx_{k-g+1}^{\alpha}v_{g-1}$, i.e., $B[g]$ is a $(k-g+1)$-limited $\alpha$-bisimulation. We complete the sufficiency.
	
	(Necessity.) By the construction of $T_{\mathcal{S}}(u,k)$ and $T_{\mathcal{S}}(v,k)$, and the definition of $B[k+1]$, $B[k+1]$ is a relation between the set of steady states of $T_{\mathcal{S}}(u,k)$ and the set of steady states of $T_{\mathcal{S}}(v,k)$, and so $B[k+1]$ is an $\alpha$-bisimulation under $\widetilde{S}$ between $T_{\mathcal{S}}(u,k)$ and $T_{\mathcal{S}}(v,k)$. Suppose that $\cup_{g\leq i\leq k+1}B[i]$ is an $\alpha$-bisimulation under $\widetilde{S}$ between $T_{\mathcal{S}}(u,k)$ and $T_{\mathcal{S}}(v,k)$, $2\leq g\leq k+1$.
	We now show that $\cup_{g-1\leq i\leq k+1}B[i]$ is an $\alpha$-bisimulation under $\widetilde{S}$ between $T_{\mathcal{S}}(u,k)$ and $T_{\mathcal{S}}(v,k)$. For any $(u_{l},v_{l})\in\cup_{g-1\leq i\leq k+1}B[i]$, $(u_{l},v_{l})\in B[l']$ is assumed for some $g-1\leq l'\leq k+1$. Then by Definition \ref{Def k-limited bisimulation vector}, $u_{l}\os{\tau}\lra p$ implies $v_{l}\os{\tau}\lra q$ such that $(p,q)\in(B[l'+1])_{\alpha}^{\dag}\subseteq(\cup_{g-1\leq i\leq k+1}B[i])_{\alpha}^{\dag}$.
	
	Without changing the analysis method, the symmetric assertion can be proved. By the induction hypothesis and Definition \ref{Def app-bisimulation}, $\cup_{g-1\leq i\leq k+1}B[i]$ is an $\alpha$-bisimulation under $\widetilde{S}$ between $T_{\mathcal{S}}(u,k)$ and $T_{\mathcal{S}}(v,k)$. The necessity also holds.
\end{proof}

From Definition \ref{Def k-limited bisimulation}, we know that the $i$-th labels two states perform and the properties of $D_u[i]$ and $D_v[i]$, i.e., whether for any distribution in $D_u[i]$ has matching distribution in $D_v[i]$, and vise versa,  $0\leq i\leq k$, determine the similarity between two states $u$ and $v$. Definition \ref{Def k-limited bisimulation vector} illustrates that if $(u,v)$ belongs to the $1$-th element of a $k$-limited $\alpha$-bisimulation vector, then $u$ and $v$ are $k$-limited $\alpha$-bisimilar. Hence it is sufficient to construct the greatest $k$-limited $\alpha$-bisimulation vector $H_k^{\alpha}(u,v)$ contained in $B_k(u,v)=(D_u[0]\times D_v[0],D_u[1]\times D_v[1],\cdots,D_u[k]\times D_v[k])$, i.e., $H_k^{\alpha}(u,v)[i]\subseteq D_u[i]\times D_v[i]$, $0\leq i\leq k$, to detect $k$-limited $\alpha$-bisimilarity between $u$ and $v$.

Given a finite NFTS $\mathcal{S}=(V,\Sigma,\delta)$, $\alpha\in[0,1]$, and $u,v\in V$. For later use, we define a $(k+1)$-vector $H_k^{\alpha}(u,v)$ with elements belonging to $\mathcal{P}(V\times V)$ as follows:
\begin{equation}\label{sequences approx}
	\begin{split}
		H_{k}^{\alpha}(u,v)[k+1]&=D_u[k]\times D_v[k];\\
		H_{k}^{\alpha}(u,v)[i]&=\{(u_i,v_i)\in D_u[i]\times D_v[i]\,|\,\\
		u_i\os{\tau}\lra p\mbox{ i}&\mbox{mplies }v_i\os{\tau}\lra q \mbox{ such that }(p,q)\in(H_{k}^{\alpha}(u,v)[i+1])_{\alpha}^{\dag},\\
		v_i\os{\tau}\lra q\mbox{ i}&\mbox{mplies } u_i\os{\tau}\lra p\mbox{ such that }
		(p,q)\in(H_{k}^{\alpha}(u,v)[i+1])_{\alpha}^{\dag}\}, \quad 1\leq i\leq k.
	\end{split}
\end{equation}

An equivalent condition for determining two states to be $k$-limited $\alpha$-bisimilar is given resorting to the vector $H_k^{\alpha}(u,v)$.
\begin{Thm}\label{Thm iff limited bisimilarity}
	{ Let $\mathcal{S}=(V,\Sigma,\delta)$ be a finitary NFTS, $u,v\in V$, $k\in\mathbb{N}$ and $\alpha\in[0,1]$. Let $H_{k}^{\alpha}(u,v)$ be the $(k+1)$-vector defined in (\ref{sequences approx}). Then $H_{k}^{\alpha}(u,v)$ is a $k$-limited $\alpha$-bisimulation vector contained in $B_k(u,v)$, and further $u\approx_k^{\alpha}v$ iff $H_{k}^{\alpha}(u,v)[1]=\{(u,v)\}$.
	}
\end{Thm}
\begin{proof}
	(Sufficiency.) By Definition \ref{Def k-limited bisimulation vector}, $H_k^{\alpha}(u,v)$ is a $k$-limited $\alpha$-bisimulation vector contained in $B_k(u,v)$. Since $H_{k}^{\alpha}(u,v)[1]=\{(u,v)\}$, by Definition \ref{Def k-limited bisimulation vector}, $u\approx_k^{\alpha}v$. The sufficiency holds.
	
	(Necessity.) We begin with the claim: for any $p\in\mathcal{R}(D_u[i])$ and $q\in\mathcal{R}(D_v[i])$, $0\leq i\leq k$, the following fact holds:
	\begin{equation*}
		\begin{split}\label{iff iff iff}
			&(p,q)\in(\approx_{k-i}^{\alpha})_{\alpha}^{\dag} \mbox{ iff }(p,q)\in(H_{k}^{\alpha}(u,v)[i+1])_{\alpha}^{\dag}.
		\end{split}
	\end{equation*}
	
	By Definition \ref{relation lifting},
	\begin{align*}
		&(p,q)\in(\approx_{k-i}^{\alpha})_{\alpha}^{\dag} \mbox{ iff } (p,q)\in(({\rm su}(p)\times {\rm su}(q))\cap\approx_{k-i}^{\alpha})_{\alpha}^{\dag} \mbox{ iff }(p,q)\in((D_u[i]\times D_v[i])\cap\approx_{k-i}^{\alpha})_{\alpha}^{\dag}=(H_{k}^{\alpha}(u,v)[i+1])_{\alpha}^{\dag},
	\end{align*}
	which implies that the claim holds.
	
	We next show the necessity. It is obvious that if $u\approx_0^{\alpha}v$, then $H_{m+1}^{\alpha}(u,v)=\{(u,v)\}$. Resorting to the correctness of the statement for $k=m$,  $m\in\mathbb{N}$, and $u\approx_{m+1}^{\alpha}v$, by Definition \ref{Def k-limited bisimulation} and the claim, $u\os{\tau}\lra p$ implies $v\os{\tau}\lra q$ such that $(p,q)\in(H_{m+1}^{\alpha}(u,v)[2])_{\alpha}^{\dag}$. Similarly, $v\os{\tau}\lra q$ implies $u\os{\tau}\lra p$ such that $(p,q)\in(H_{m+1}^{\alpha}(u,v)[2])_{\alpha}^{\dag}$. Hence $(u,v)\in H_{m+1}^{\alpha}(u,v)[1]\subseteq\{(u,v)\}$, i.e., $H_{m+1}^{\alpha}(u,v)[1]=\{(u,v)\}$.
\end{proof}

\subsection{The degree of $k$-limited similarity}
This subsection is devoted to an algorithm, named  {\it Limited similarity}, to calculate the degree of $k$-limited similarity between two states in an NFTS.

Limited similarity is divided into three parts: 1-limited similarity (Algorithm \ref{alg:1}), limited bisimulation vector computation (Algorithm \ref{alg:2}), degree of $k$-limited similarity computation (Algorithm \ref{alg:3}). Each part is described as follows.

{\it 1-limited similarity} (Part 1). Let $\alpha\in[0,1]$, $u,v\in V$, and $R\subseteq V\times V$. The degree of 1-limited similarity between $u$ and $v$ with respect to $R$ is defined as the greatest $\alpha$ satisfying that $u\os{\tau}\lra p$ implies $v\os{\tau}\lra q$ that satisfies $(p,q)\in R_{\alpha}^{\dag}$, and that $v\os{\tau}\lra q$ implies $u\os{\tau}\lra p$ that satisfies $(p,q)\in R_{\alpha}^{\dag}$. In terms of the notion, Algorithm \ref{alg:1} aims to compute the degree of 1-limited similarity between $u$ and $v$ with respect to $R$, i.e., $\bigwedge_{u\os{\tau}\lra p}\bigvee_{v\os{\tau}\lra q}\widetilde{S}(p,q,R)\wedge\bigwedge_{v\os{\tau}\lra q}\bigvee_{u\os{\tau}\lra p}\widetilde{S}(p,q,R)$.

{\it Limited bisimulation vector computation} (Part 2). Let $B(u,v)$ be a $(k+1)$-vector with $B(u,v)[i+1]\subseteq D_u[i]\times D_v[i]$, $0\leq i\leq k$. Since $\mathbf{Vec}((u,v),k,\alpha,B(u,v))[k+1]=B(u,v)[k+1]$ is the $0$-limited $\alpha$-bisimilarity contained in $B(u,v)[k+1]$. Steps 4-8 of Algorithm \ref{alg:2} compute the $(k-i+1)$-limited $\alpha$-bisimilarity $\mathbf{Vec}((u,v),$ $k,\alpha,B(u,v))[i]$ contained in $B(u,v)[i]$, $1\leq i\leq k+1$. Hence, Algorithm \ref{alg:2} computes the greatest $k$-limited $\alpha$-bisimulation vector contained in $B(u,v)$.

{\it Degree of $k$-limited similarity computation} (Part 3). Let $B'(u,v)$ be a $(k+1)$-vector that satisfies $B'(u,v)[i+1]\subseteq D_u[i]\times D_v[i]$, $0\leq i\leq k$. The main idea of this part is that if $\alpha'\leq\alpha''$, then the greatest $k$-limited $\alpha''$-bisimulation vector contained in $B'(u,v)$ is contained in the greatest $k$-limited $\alpha'$-bisimulation vector which is contained in $B'(u,v)$.

Assume that $\mathbf{BIS}((u,v),k)=\alpha_m$. By Algorithm \ref{alg:2}, we know that $T_{j+1}=\mathbf{Vec}((u,v),k,\alpha_j,B_{j+1})$ is the greatest $k$-limited $\alpha_j$-bisimulation vector contained in $B_{j+1}$, and then by Step 5 of Algorithm \ref{alg:3} and Proposition \ref{Degree of bismilation vector}, $\alpha_{j+1}\geq\alpha_j$, $0\leq j\leq m-1$. Let
$\epsilon=\alpha_{0}=\alpha_{l_0}<\alpha_1=\cdots=\alpha_{l_1-1}<\alpha_{l_1}\cdots\alpha_{l_n-1}<\alpha_{l_n}=\cdots=\alpha_{m}$, where $\alpha_m$ is the greatest number such that $T_m[1]=\{(u,v)\}$.
By Step 5 of Algorithm \ref{alg:3}, $T_{l_i}$ is the $k$-limited $\alpha_{l_i}$-bisimilarity vector contained in $B(u,v)$. 
By Step 8 of Algorithm \ref{alg:3}, $\mathbf{BIS}((u,v),k)$, i.e., $\alpha_m$, is the greatest real number such that $u\approx_k^{\alpha_m}v$. 

\begin{algorithm}[h]
	\caption{Compute the degree of 1-limited similarity.}
	\label{alg:1}
	\KwIn{NFTS $\mathcal{S}=(V,\Sigma,\delta)$, $k\in\mathbb{N}$, $u',v'\in V$ and $R\subseteq V\times V$.}
	\KwOut{$\mathbf{Deg}((u',v'),R)$, the degree of 1-limited similarity  between $u'$ and $v'$ with respect to $R_1$.}
	\textbf{Initialize} matrix $M$ to be zero\;
		\For{$u'\os{\tau}\lra p$}{
			\For{$v'\os{\tau}\lra q$}{
				\For{$u\in V$}{
					$N_1[u]=p(u)\rra q(\orra{R}_{u})$\;
					$N_2[u]=q(u)\rra p(\olla{R}_{u})$\;
				}
				$M((\tau,p),(\tau,q),R)=\bigwedge_{u\in V}(N_1[u]\wedge N_2[u])$\;
			}
		}
		\Return $\bigwedge_{u\os{\tau}\lra p}\bigvee_{v\os{\tau}\lra q}M((\tau,p),(\tau, q),R)\wedge\bigwedge_{v\os{\tau}\lra q}\bigvee_{u\os{\tau}\lra p}M((\tau,p),(\tau,q),R)$
	\end{algorithm}
	
	\begin{algorithm}[h]
		\caption{Compute limited bisimulation vector.}
		\label{alg:2}
		\KwIn{NFTS $\mathcal{S}=(V,\Sigma,\delta)$, $k\in\mathbb{N}$, $\alpha\in[0,1]$, $u,v\in V$, and a $(k+1)$-vector $B(u,v)$ with $B(u,v)[i+1]\subseteq D_u[i]\times D_v[i]$, $0\leq i\leq k$.}
		\KwOut{$\mathrm{\mathbf{Vec}}((u,v),k,\alpha,B(u,v))$, the greatest $k$-limited $\alpha$-bisimulation vector contained in $B(u,v)$.}%
		\textbf{Initialize} $H(u,v)[i]=B(u,v)[i]$, $1\leq i\leq k+1$, $j=k$\;
		\Repeat
		{$j=0$}
		{
			$j:=j-1$\;
			\For{$\mathbf{all}$ $(u_j,v_j)\in H(u,v)[j]$}{
				\If{$\mathbf{Deg}((u_j,v_j),H(u,v)[j+1])<\alpha$}{
					$H(u,v)[j]:=H(u,v)[j]\backslash \{(u,v),(v,u)\}$\;
				}
			}
		}
		\Return $(H(u,v)[1],H(u,v)[2],\cdots,H(u,v)[k+1])$
	\end{algorithm}
	
	\begin{algorithm}
		\caption{Compute the degree of $k$-limited similarity.}
		\label{alg:3}
		\KwIn{NFTS $\mathcal{S}=(V,\Sigma,\delta)$, $k\in\mathbb{N}$ and $u,v\in V$.}
		\KwOut{$\textbf{BIS}((u,v),k)$, the degree of $k$-limited similarity between two states $u$ and $v$.}
		\textbf{Initialize} arrays $T_i$ and $B_i$ to be zero, $i\in\mathbb{N}$ and $\epsilon$ be the minimum non-negative real number, $B(u,v)[i+1]=D_u[i]\times D_v[i]$, $0\leq i\leq k$, $T_1=\mathbf{Vec}((u,v),k,\epsilon,B(u,v))$, $\alpha_0=0$ and $j=0$\;
		\Repeat
		{$T_j[1]=\emptyset$}
		{
			$j=j+1$\;
			\If{$T_j[1]=\{(u,v)\}$}{
				$\alpha_{j}=\bigwedge_{1\leq i\leq k}\bigwedge_{r\in T_j[i]}\mathbf{Deg}(r,T_j[i+1])$\;
				$B_{j+1}=T_j-\cup_{1\leq i\leq k}\{r\in T_j[i]\mid\mathbf{Deg}(r,T_j[i+1])\leq\alpha_j\}$\;
		    	$T_{j+1}=\mathbf{Vec}((u,v),k,\alpha_j,B_{j+1})$\;}}
		\Return $\alpha_{j-1}$
	\end{algorithm}
	
	By the degree of 1-limited similarity with respect to a given relation, Algorithm \ref{alg:1} correctly computes the degree of 1-limited similarity between $s'$ and $t'$ with respect to $R$.
	
	Before presenting the correctness of computing the degree of similarity, some auxiliary lemma and proposition are provided.
	
	\begin{Prop}\label{Correctness alg:2}
		{Algorithm \ref{alg:2} correctly computes the greatest $k$-limited $\alpha$-bisimulation vector contained in $B(u,v)$.
		}
	\end{Prop}
	\begin{proof}
		Let $A=\mathbf{Vec}((u,v),k,\alpha,B(u,v))$ for convenience. By Algorithm \ref{alg:2}, $A[k+1]=H(u,v)[k+1]=B(u,v)[k+1]$, and so $A[k+1]$ is the 0-limited $\alpha$-bisimilarity contained in $B(u,v)[k+1]$. Under the condition that $A[i]$ is the $(k-i+1)$-limited $\alpha$-bisimilarity contained in $B(u,v)[i]$, where $2\leq i\leq k$, we now show that $A[i-1]$ is the $(k-i+2)$-limited $\alpha$-bisimilarity contained in $B(u,v)[i-1]$. For any $(u_{i-1},v_{i-1})\in A[i-1]$, by Steps 4--6, we have that $\mathbf{Deg}((u_{i-1},v_{i-1}),A[i])\geq\alpha$, i.e.,
		\begin{align*}
			\bigwedge_{u_{i-1}\os{\tau}\lra p}\bigvee_{v_{i-1}\os{\tau}\lra q}M((\tau,p),(\tau,q),A[i])\wedge\bigwedge_{v_{i-1}\os{\tau}\lra q}\bigvee_{u_{i-1}\os{\tau}\lra p}M((a,p),(\tau,q),A[i])\geq\alpha.
		\end{align*}
		The inequality $\bigwedge_{u_{i-1}\os{\tau}\lra p}\bigvee_{v_{i-1}\os{\tau}\lra q}M((\tau,p),(\tau,q),A[i])\geq\alpha$ means that $u_{i-1}\os{\tau}\lra p$ implies $v_{i-1}\os{\tau}\lra q$ such that $M((\tau,p),$ $(\tau,q),A[i])\geq\alpha$, i.e., $(p,q)\in(A[i])_{\alpha}^{\dag}$ by Definition \ref{relation lifting} and Algorithm \ref{alg:1}. In the similar way, $\bigwedge_{v_{i-1}\os{\tau}\lra q}\bigvee_{u_{i-1}\os{\tau}\lra p}M((\tau,p),$ $(\tau,q),$ $A[i])\geq\alpha$ means that for any $v_{i-1}\os{\tau}\lra q$, there exists $u_{i-1}\os{\tau}\lra p$ such that $(p,q)\in (A[i])_{\alpha}^{\dag}$. Hence $A[i-1]$ is a $(k-i+2)$-limited $\alpha$-bisimulation. Since $A[i]$ is the $(k-i+1)$-limited $\alpha$-bisimilarity contained in $B(u,v)[i]$, $A[i]$ is the ($k-i$)-limited $\alpha$-bisimilarity contained in $B(u,v)[i]$, $A[i-1]$ is the ($k-i+2$)-limited $\alpha$-bisimilarity contained in $B(u,v)[i-1]$ by Lemma \ref{Prop lifted relation} (2) and Definition \ref{Def k-limited bisimulation}. By Definition \ref{Def k-limited bisimulation vector}, $\mathbf{Vec}((u,v),k,\alpha,B(u,v))$ is the $k$-limited $\alpha$-bisimilarity vector contained in $B(u,v)$.
	\end{proof}
	
	\begin{Lem}\label{algorithm 3 maximum limited bisimilarity}
		{ Let $\{\alpha_j\}_{0\leq j\leq m}$ be the sequence of elements in [0,1] derived from Algorithm \ref{alg:3}, where $\alpha_{0}=\epsilon$ and $\mathbf{BIS}((u,v),k)=\alpha_m$. If $\alpha_g<\alpha_{g+1}$ for some $0\leq g\leq m-1$, then $T_{g+1}=\mathbf{Vec}((u,v),k,\alpha_g,B_{g+1})$ is the greatest $k$-limited $\alpha_{g+1}$-bisimulation vector contained in $B(u,v)$.
		}
	\end{Lem}
	\begin{proof}
		We start with the claim: for the sequence $\{\alpha_j\}_{0\leq j\leq m}$, $\alpha_j\leq\alpha_{j+1}$ holds, $0\leq j\leq m$.
		
		By Proposition \ref{Correctness alg:2}, $T_{j+1}$ is a $k$-limited $\alpha_j$-bisimulation vector contained in $B(u,v)$, $0\leq j\leq m-1$. Since $\alpha_{j+1}=\bigwedge_{1\leq i\leq k}\bigwedge_{r\in T_{j+1}[i]}$ $\mathbf{Deg}(r,T_{j+1}[i+1])=Q(T_{j+1})$, by Proposition \ref{Degree of bismilation vector}, $\alpha_j\leq\alpha_{j+1}$. The claim holds.
		
		Let
		\begin{align*}
			&\alpha_{l_0}<\alpha_1=\cdots=\alpha_{l_1-1}<\alpha_{l_1}=\alpha_{l_1+1}=\cdots\alpha_{l_2-1}<\alpha_{l_2}\cdots\alpha_{l_{n}-1}\leq\alpha_{l_n}=\alpha_{l_n+1}=\cdots=\alpha_m,
		\end{align*}
		where $l_0=0$. We now prove that $T_{l_i}$ is the $k$-limited $\alpha_{l_i}$-bisimilarity vector contained in $B(u,v)$, $0\leq i\leq n$. Since $T_{l_0}=\mathbf{Vec}((u,v),k,\alpha_{0},B(u,v))$, by Proposition \ref{Correctness alg:2}, $T_{l_0}$ is the $k$-limited $\alpha_{0}$-bisimilarity vector contained in $B(u,v)$. By Proposition \ref{Degree of bismilation vector} and $\alpha_{l_0}=\bigwedge_{1\leq i\leq k}\bigwedge_{r\in T_{l_0}[i]}\mathbf{Deg}(r,T_{l_0}[i+1])=Q(T_{l_0})$, $T_{l_0}$ is the $k$-limited $\alpha_{l_0}$-bisimilarity vector contained in $B(u,v)$. Assume that $T_{l_p}$ is the $k$-limited $\alpha_{l_{p}}$-bisimilarity vector contained in $B(u,v)$, i.e., $T_{l_{g}}=\mathbf{Vec}((u,v),k,\alpha_{l_g},B(u,v))$, $0\leq g\leq n-1$. We now show that $T_{l_{g+1}}$ is the $k$-limited $\alpha_{l_{g+1}}$-bisimilarity vector contained in $B(u,v)$, i.e., $T_{l_{g+1}}=\mathbf{Vec}((u,v),k,\alpha_{l_{g+1}},B(u,v))$. Let $T_{l_{g+1}}'=\mathbf{Vec}((u,v),k,\alpha_{l_{g+1}},B(u,v))$. By Propositions \ref{Degree of bismilation vector} and \ref{Correctness alg:2}, we know that
		\begin{align}\label{subseteq1}
			\begin{split}
				T_{l_{g+1}}=&\mathbf{Vec}((u,v),k,\alpha_{l_{g+1}-1},B_{l_{g+1}}(u,v))\\
				=&\mathbf{Vec}((u,v),k,\alpha_{l_{g+1}},B_{l_{g+1}}(u,v))\\
				\subseteq&\mathbf{Vec}((u,v),k,\alpha_{l_{g+1}},B(u,v))\\
				=&T_{l_{g+1}}'.
			\end{split}
		\end{align}
		By the known condition that $T_{l_p}$ is the $k$-limited $\alpha_{l_{p}}$-bisimilarity vector contained in $B(u,v)$ and the claim,
		\begin{align*}
			T_{l_{g+1}}'=&\mathbf{Vec}((u,v),k,\alpha_{l_{g+1}},B(u,v))\\
			\subseteq&\mathbf{Vec}((u,v),k,\alpha_{l_{g}},B(u,v))\\
			=&T_{l_{g}}.
		\end{align*}
		Then $T_{l_{g+1}}'$ is the $k$-limited $\alpha_{l_{g+1}}$-bisimilarity vector contained in $T_{l_{g}}$, i.e., $T_{l_{g+1}}'=\mathbf{Vec}((u,v),k,\alpha_{l_{g+1}},T_{l_g})$. Assume that $T_{l_{g+1}}'=\mathbf{Vec}((u,v),k,\alpha_{l_{g+1}},T_{l_g+h})$, $0\leq h\leq l_{g+1}-l_{g}-1$. Then we prove that $T_{l_{g+1}}'=\mathbf{Vec}((u,v),k,\alpha_{l_{g+1}},T_{l_g+h+1})$.
		
		By Definition \ref{Def k-limited bisimulation vector},
		\begin{align*}
			T_{l_{g+1}}'=&\mathbf{Vec}((u,v),k,\alpha_{l_{g+1}},T_{l_{g}+h})\\
			\subseteq&T_{l_{g}+h}-\cup_{1\leq i\leq k}\{r\in T_{l_{g}+h}[i]\mid\mathbf{Deg}(r,T_{l_g+h}[i+1])\leq\alpha_{l_{g}+h}\}\\
			=&B_{l_{g}+h+1}.
		\end{align*}
		Then $T_{l_{g+1}}'$ is also the $k$-limited $\alpha_{l_{g+1}}$-bisimilarity vector contained in $B_{l_g+g+1}$, i.e., $T_{l_{g+1}}'=\mathbf{Vec}((u,v),k,\alpha_{l_{g+1}},$ $B_{l_g+h+1})$. Hence  $T_{l_{g+1}}'=\mathbf{Vec}((u,v),k,\alpha_{l_{g+1}},$ $B_{l_{g+1}})$.
		
		Since $\alpha_{l_g}<\alpha_{l_{g+1}}=\alpha_{(l_{g+1}-1)}$, by Proposition \ref{Prop k-limited bisimulation},
		\begin{align}\label{subseteq2}
			\begin{split}
				T_{l_{g+1}}'=&\mathbf{Vec}((u,v),k,\alpha_{l_{g+1}},B_{l_{g+1}})\\
				\subseteq&\mathbf{Vec}((u,v),k,\alpha_{(l_{g+1}-1)},B_{l_{g+1}})\\
				=&T_{l_{g+1}}.
			\end{split}
		\end{align}
		Combining (\ref{subseteq1}) and (\ref{subseteq2}), we get that $T_{l_{g+1}}'= T_{l_{g+1}}$. By Proposition \ref{Correctness alg:2}, $T_{l_{g+1}}$ is the $k$-limited $\alpha_{l_{g+1}}$-bisimilarity vector contained in $B(u,v)$.
	\end{proof}
	
 With the following proof, Algorithm \ref{alg:3} is correct.  
	\begin{Thm}\label{correctness alg:3}
		{Let $\mathcal{S}=(V,\Sigma,\delta)$ be an NFTS, $u,v\in V$, $k\in\mathbb{N}$ and $\alpha\in[0,1]$. Algorithm \ref{alg:3} correctly computes the degree of $k$-limited similarity between $u$ and $v$.
		}
	\end{Thm}
	\begin{proof}
		Supposing $\mathbf{BIS}((u,v),k)=\alpha_m$, $m\in\mathbb{N}$. Then we have that $T_{m+1}[1]\neq\{(u,v)\}$. Proving that $u\approx_k^{\alpha}v$ iff $\alpha\leq\alpha_m$ is sufficient to prove the statement.
		
		(Sufficiency.) By Algorithm \ref{alg:3}, $T_m[1]=\{(u,v)\}$, and so
		\begin{align*}
			\alpha_m=\bigwedge_{1\leq i\leq k}\bigwedge_{r\in T_m[i]}\mathbf{Deg}(r,T_m[i+1]).
		\end{align*}
		By Proposition \ref{Degree of bismilation vector}, $T_m$ is a $k$-limited $\alpha_m$-bisimulation vector contained in $B_m$, which implies that
		$T_m\subseteq$ $\mathbf{Vec}((u,v),k,\alpha_{m},$ $B_m)$
		by Proposition \ref{Correctness alg:2}. Hence
		$\mathbf{Vec}((u,v),k,\alpha_{m},B_m)[1]=\{(u,v)\}$.
		By Theorem \ref{Thm iff limited bisimilarity}, $u\approx_k^{\alpha_m}v$. We see that $u\approx_k^{\alpha}v$ from $\alpha\leq\alpha_m$ and Proposition \ref{Prop k-limited bisimulation} (1). The sufficiency holds.
		
		(Necessity.) Assume that $u\approx_k^{\alpha}v$ for some $\alpha>\alpha_m$. Since $u\approx_k^{\alpha}v$, by Theorem \ref{Thm iff limited bisimilarity},
		\begin{align}\label{Vec alpha}
			\mathbf{Vec}((u,v),k,\alpha,B(u,v))[1]=\{(u,v)\}.
		\end{align}
		From $\alpha>\alpha_{m}$, Definition \ref{Def k-limited bisimulation vector} and Proposition \ref{Prop k-limited bisimulation}, we get that
		\begin{align*}
			\mathbf{Vec}((u,v),k,\alpha,B(u,v))\subseteq\mathbf{Vec}((u,v),k,\alpha_m,B(u,v)).
		\end{align*}
		Let
		\begin{align*}
			&\epsilon=\alpha_{l_0}<\alpha_1=\cdots=\alpha_{l_1-1}<\alpha_{l_1}=\alpha_{l_1+1}=\cdots=\alpha_{l_2-1}<\alpha_{l_2}\cdots\alpha_{l_{n}-1}\leq\alpha_{l_n}=\alpha_{l_n+1}=\cdots=\alpha_m,
		\end{align*}
		where $l_0=0$. By Lemma \ref{algorithm 3 maximum limited bisimilarity},
		\begin{align*}
			\mathbf{Vec}((u,v),k,\alpha,B(u,v))\subseteq&\mathbf{Vec}((u,v),k,\alpha_m,B(u,v))\\
			=&\mathbf{Vec}((u,v),k,\alpha_{l_{n}},B(u,v))\\
			=&T_{l_n}.
		\end{align*}
		Assume that $\mathbf{Vec}((u,v),k,\alpha,B(u,v))\subseteq T_{l_{n}+g}$, $0\leq g\leq m-l_n-1$. We show that $\mathbf{Vec}((u,v),k,\alpha,B(u,v))\subseteq T_{l_{n}+g+1}$. We have
		\begin{align*}
			\mathbf{Vec}((u,v),k,\alpha,B(u,v))
			&=\mathbf{Vec}((u,v),k,\alpha,T_{l_{n}+g}) \quad (\mbox{by Definition \ref{Def k-limited bisimulation vector}})\\
			&\subseteq T_{l_{n}+g}-\cup_{1\leq i\leq k}\{r\in T_{l_{n}+g}[i]\mid\mathbf{Deg}(r,T_{l_{n}+g}[i+1])\leq\alpha_{l_{n}+g}\}\\
			&=B_{l_{n}+g+1}\quad (\mbox{by Algorithm \ref{alg:3}}),
		\end{align*}
		and so,
		\begin{align*}
			\mathbf{Vec}((u,v),k,\alpha,B(u,v))
			&=\mathbf{Vec}((u,v),k,\alpha,B_{l_{n}+g+1}) \quad (\mbox{by Definition \ref{Def k-limited bisimulation vector}})\\
			&\subseteq\mathbf{Vec}((u,v),k,\alpha_{l_n+g},B_{l_{n}+g+1}) \quad (\mbox{by Proposition \ref{Prop k-limited bisimulation}})\\
			&=T_{l_{n}+g+1} \quad (\mbox{by Algorithm \ref{alg:3}}).
		\end{align*}
		Hence $\mathbf{Vec}((u,v),k,\alpha,B(u,v))\subseteq T_{m}$, and further, we have
		\begin{align*}
			\mathbf{Vec}((u,v),k,\alpha,B(u,v))
			&=\mathbf{Vec}((u,v),k,\alpha,T_{m})\\
			&\subseteq T_{m}-\cup_{0\leq i\leq k-1}\{r\in T_{m}[i]\mid\mathbf{Deg}(r,T_{m}[i+1])\leq\alpha_{m}\}\\
			&=B_{m+1}.
		\end{align*}
		By Definition \ref{Def k-limited bisimulation vector} again,
		\begin{align*}
			\mathbf{Vec}((u,v),k,\alpha,B(u,v))=&\mathbf{Vec}((u,v),k,\alpha,B_{m+1})\\
			\subseteq&\mathbf{Vec}((u,v),k,\alpha_m,B_{m+1})\\
			=&T_{m+1}.
		\end{align*}
		By (\ref{Vec alpha}), we obtain that
		$\{(u,v)\}=\mathbf{Vec}((u,v),k,\alpha,B(u,v))[1]$ $\subseteq T_{m+1}[1]\subseteq\{(u,v)\}$,
		i.e., $T_{m+1}[1]=\{(u,v)\}$. This contradicts $T_{m+1}[1]\neq\{(u,v)\}$. Therefore, if $u\approx_k^{\alpha}v$, then $\alpha\leq\alpha_m$. We complete the proof.
	\end{proof}
	
	We now investigate the time complexities of Algorithms \ref{alg:1}, \ref{alg:2} and \ref{alg:3}. Let $\left|\lra\right|=\bigvee_{u\in V}|\{p\in\mathcal{R}(V)\mid u\os{\tau}\lra p\}|$. 
	
	{\it Complexity analysis}\quad In Algorithm \ref{alg:1}, Steps 4--6 take $|V|^2$. Steps 3--7 form an inner loop and the loop repeats $\left|\lra\right|$, and Steps 2--7 form an outer loop and the loop also repeats $\left|\lra\right|$. In addition, Step 7 costs $|V|$. Hence Steps 2--7 costs $(|V|^2+|V|)\cdot\left|\lra\right|^2$. Similarly, Steps 8--13 costs $(|V|^2+|V|)\cdot\left|\lra\right|^2$.
	Therefore, Algorithm \ref{alg:1} costs $O(2|V|^2\cdot\left|\lra\right|^2)$.
	
	For Algorithm \ref{alg:2}, Steps 5 and 6 cost $O(2|V|^2\cdot\left|\lra\right|^2)$ by Algorithm \ref{alg:1}. $B(u,v)[i]$ contains at most $|V|^2$ elements, $1\leq i\leq k+1$, then the inner loop at Steps 4--6 repeats $|V|^2$. The loop at Steps 2--7 repeats $k$, and hence Algorithm \ref{alg:2} is in $O(2k|V|^4\cdot\left|\lra\right|^2)$.
	
	In Algorithm \ref{alg:3}, since $T_j[i]$ has at most $|V|^2$ elements, by Algorithm \ref{alg:1}, Step 5 is in $O(2k\cdot|V|^4\cdot\left|\lra\right|^2))$, and Step 6 also costs $O(2k|V|^4\cdot\left|\lra\right|^2))$. By Algorithm \ref{alg:2}, Step 7 costs $O(2k|V|^4\cdot\left|\lra\right|^2)$. Since $|V|$ is finite, $|T_0[i]|\leq|V|^2$, $1\leq i\leq k$. Then Steps 2--8 forms a loop and the loop repeats at most $k|V|^2$. Hence Algorithm \ref{alg:3} costs $O(2k^2|V|^6\cdot\left|\lra\right|^2)$.
	
	An example is given to illustrate the algorithm Limited similarity.
	\begin{Eg}
		{\rm Let $\mathcal{S}=(V,\Sigma,\delta)$ be the NFTS presented in Fig. \ref{Figure2}, $\mathcal{L}$ be a G\"{o}del algebra, $k=3$. Let $u_0=u$ and $v_0=v$. By Algorithms \ref{alg:1} and \ref{alg:2},
			\begin{align*}
				&T_1[4]=B(u,v)[3]=D_u[3]\times D_t[3]=\{(u_{23},v_{53})\},\\
				&T_1[3]=\{(u_{32},v_{32}),(u_{42},v_{42})\},\\
				&T_1[2]=\{(u_{11},v_{11}),(u_{21},v_{21})\},\\
				&T_1[1]=\{(u_{00},v_{00})\},
			\end{align*}
			and so $T_1=(\{(u_{00},v_{00})\},\{(u_{11},v_{11}),u_{21},v_{21})\},\{(u_{32},v_{32}),$ $(u_{42},v_{42})\},\{(u_{23},v_{53})\})$. By Algorithm \ref{alg:1}, we obtain that
			\begin{align*}
				\alpha_1=&0.3,\\
				B_{2}=&T_1-\cup_{1\leq i\leq 3}\{r\in T_1[i]\mid\mathbf{Deg}(r,T_1[1]\leq\alpha_1)\}\\
				=&(\{(u_{00},t_{00})\},\{(u_{11},v_{11}),(u_{21},v_{21})\},\{(u_{42},v_{42})\}),
			\end{align*}
			By Algorithm \ref{alg:2}, $T_2=\mathbf{\emptyset}$. By Theorem \ref{correctness alg:3}, $\mathbf{BIS}((u,v),3)=\alpha_1=0.3$.
		}
	\end{Eg}
	
	\section{Logical characterization}
	The relationship between $k$-limited bisimulation and $\alpha$-bisimulation under $\widetilde{S}$ is looked into, and then the limited bisimulation is characterized resorting to a suitable modal logic.
	
	\subsection{Relationship with $\alpha$-bisimulation under $\widetilde{S}$}
	
	When determining the $k$-limited $\alpha$-bisimilarity between two states, it is sufficient to investigate the $\alpha$-bisimilarity under $\widetilde{S}$ between them in their $k$-neighbouring subsystems (subgraphs).
	\begin{Prop}\label{Prop s sim t}
		{Let $\mathcal{S}=(V,\Sigma,\delta)$ be an NFTS, $u,v\in V$, $k\in\mathbb{N}$ and $\alpha\in[0,1]$. If $u\sim^{\alpha}v$ in $\mathcal{S}(u,k)$ and $\mathcal{S}(v,k)$, then $u\approx_k^{\alpha}v$.
		}
	\end{Prop}
	\begin{proof}
		With the case of $k=0$ (trivial) holds, $k=m$, and $u\sim^{\alpha}v$ in $\mathcal{S}(u,m+1)$ and $\mathcal{S}(v,m+1)$, we explain the case of $k=m+1$ also holds. In $\mathcal{S}(u,m+1)$ and $\mathcal{S}(v,m+1)$, since $u\sim^{\alpha}v$, $u\os{\tau}\lra p$ implies $v\os{\tau}\lra q$ such that $(p,q)\in(\sim^{\alpha})_{\alpha}^{\dag}$. By Definition \ref{Def app-bisimulation}, for any $(u',v')\in({\rm su}(p)\times {\rm su}(q))\cap\sim^{\alpha}$, $u'\sim^{\alpha}v'$ in $\mathcal{S}(u',m)$ and $\mathcal{S}(v',m)$, and so $u'\approx_{m}^{\alpha}v'$ by known conditions. Hence ${\rm su}(p)\times {\rm su}(q))\cap\sim^{\alpha}\subseteq\approx_{m}^{\alpha}$, and so $(p,q)\in(\approx_{m}^{\alpha})_{\alpha}^{\dag}$ by Definition \ref{relation lifting}. Similarly, the symmetric case can be analyzed. Therefore, $u\approx_{m+1}^{\alpha}v$. We complete the proof.
	\end{proof}
	
	Two states $u$ and $v$ are not $\alpha$-bisimilar under $\widetilde{S}$ in their $k$-neighbouring subsystems $\mathcal{S}(u,k)$ and $\mathcal{S}(v,k)$ does not mean that they are not $k$-limited $\alpha$-bisimilar. This fact can be explained by the following example.
	\begin{Eg}
		{\rm Consider the NFTS $\mathcal{S}=(V,\Sigma,\delta)$ provided in Fig. \ref{Figure2}, where $V=\{u,u_1,u_2,u_3,u_4,v,v_1,v_2,v_3,v_4,v_5\}$. If $\mathcal{L}$ is a {\L}ukasiewicz algebra, since there is no binary relation $R\subseteq V\times V$ containing $(u,v)$ to be a $0.8$-bisimulation under $\widetilde{S}$, $u$ and $v$ are not $0.8$-bisimilar under $\widetilde{S}$ in $\mathcal{S}(u,3)$ and $\mathcal{S}(v,3)$ by Definition \ref{Def app-bisimulation}, while $\{(u,v),(u_1,v_1),(u_2,v_2),(u_3,v_3),$ $(u_4,v_4),$ $(u_2,v_5)\}$ is a 3-limited $0.8$-bisimulation, i.e., $u$ and $v$ are $3$-limited $0.8$-bisimilar by Definition \ref{Def k-limited bisimulation}.
		}
	\end{Eg}
	
	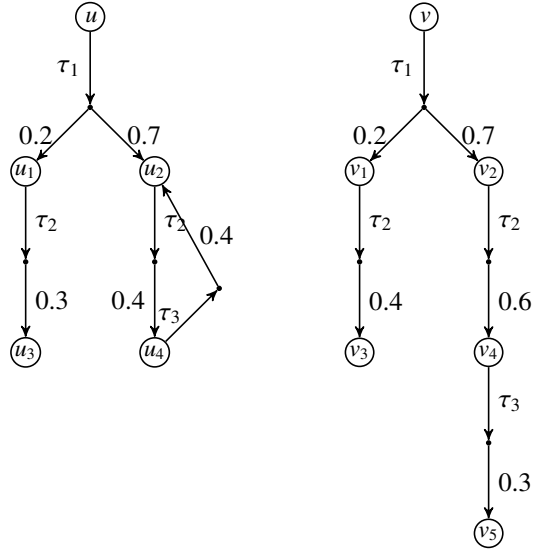
\begin{figure}[h] 
		\tikzstyle{point}=[coordinate,on grid]
		\tikzstyle{solid node}=[circle,draw,inner sep=0.5,fill=black]
		\tikzstyle{state}=[circle,draw,inner sep=0.5,scale=1,font=\bfseries\small]
		\centering
		
		\begin{tikzpicture}[->,>=stealth',shorten >=0pt,auto,node distance=1.2cm,
			semithick]
			\node[state]         (B0)  {$\;u\,$};
			\node[below of =B0] (o2) { };
			\node[solid node]         (C1) [below of=B0] {};
			\node[state]         (C2) [below left of=C1] {$u_1$};
			\node[state]         (D2)  [below right of=C1] {$u_2$};
			
			\node[solid node]         (C3) [below of=C2] {};
			\node[solid node]         (D3) [below of=D2] {};
			\node[state]         (E3) [below of=C3] {$u_3$};
			\node[state]         (E4) [below of=D3] {$u_4$};
			\node[solid node]    (F1) [above right of=E4] {};

			\path (B0) edge node [left]{$\tau_1$} (C1);
			\path (C1) edge node [left] {$0.2$} (C2);
			\path (C1) edge node [right] {$0.7$} (D2);
			\path (C2) edge node {$\tau_2$} (C3);
			\path (D2) edge node {$\tau_2$} (D3);
			\path (C3) edge node {$0.3$} (E3);
			\path (D3) edge node [left] {$0.4$} (E4);
			\path (E4) edge node [left] {$\tau_3$} (F1);
			\path (F1) edge node  [right] {$0.4$} (D2);
			
			
			\node[state]         (C1') [right=4cm of B0] {$\;v\,$};
			\node[below of =C1'] (o2) { };
			
			\node[solid node]    (D1') [below of=C1'] {};
			\node[state]         (E1') [below left of=D1'] {$v_1$};
			\node[state]         (E2') [below right of=D1'] {$v_2$};
			\node[solid node]    (F1') [below of=E1'] {};
			\node[state]         (G1') [below of=F1'] {$v_3$};
			\node[solid node]    (F2') [below of=E2'] {};
			\node[state]         (G2') [below of=F2'] {$v_4$};
			
			\node[solid node]     (H1') [below of=G2'] {};
			\node[state]         (I1') [below of=H1'] {$v_5$};
			
			\path (C1') edge node [left]{$\tau_1$} (D1');
			\path (D1') edge node [left]{$0.2$} (E1');
			
			\path (D1') edge node [right] {$0.7$} (E2');
			\path (E1') edge node [right] {$\tau_2$} (F1');
			\path (E2') edge node [right] {$\tau_2$} (F2');
			\path (F1') edge node [right] {$0.4$} (G1');
			\path (F2') edge node [right] {$0.6$} (G2');
			\path (G2') edge node [right] {$\tau_3$} (H1');
			\path (H1') edge node [right] {$0.3$} (I1');
		\end{tikzpicture}
		\caption{States $u$ and $v$ are approximate bisimilar.} \label{Figure2}
	\end{figure}
	
	Let $\mathcal{S}=(V,\Sigma,\delta)$. In order to investigate the equivalent condition of $k$-limited $\alpha$-bisimilarity and $\alpha$-bisimilarity under $\widetilde{S}$, {\it $k$-neighbouring path induced subsystem} \cite{QiaoandFeng(20231)} is needed. We recall its construction for convenience.
	
   $k$-neighbouring path induced subsystem, denoted by $T_{\mathcal{S}}(u,k)=(V',\Sigma',\delta')$, $u\in V$ and $k\in\mathbb{N}$, is an NFTS via $\mathcal{S}$ which is constructed as follows.
	
	We denote $V$ as $\{u_1,u_2,\cdots,u_n\}$, $n\in\mathbb{N}$, and define array $D_{u}$ by
	\begin{align}\label{set of k transitions}
		D_{u}[0]=\{u\}, \quad D_{s}[i+1]=(\overrightarrow{R_{\delta}})_{D_{u}[i]},\quad 0\leq i\leq k,
	\end{align}
	where $D_{u}[i]$ is the set of states reached by $u$ after $i$ transitions. If $D_{u}[i]=\{u_{j_1},u_{j_2},\cdots,u_{j_m}\}$, we relabel $u_{j_p}$ by $u_{j_pi}$, $1\leq p\leq m$, and then let $D'_{u}[i]=\{u_{j_1i},u_{j_2i},\cdots,u_{j_mi}\}$, i.e., $D_{u}[i]$ is relabeled by $D'_{u}[i]$.
	
	We define $V'=\cup_{0\leq i\leq k}(D_u'[i]\cup D_v'[i])$ and $\Sigma'=\{\tau\in \Sigma\mid\exists u'\in V' \mbox{ such that } \delta(u',\tau)\neq\emptyset\}$. Next, define $\delta'\subseteq V'\times\Sigma'\times\mathcal{R}(V')$ by
	\[
	\begin{aligned}
			\delta'(u',\tau)=
			\{p'\in\mathcal{R}(V)\mid p'(u')=\left\{\begin{array}{ll}
				p(u_{j_p}), & \textrm{if $u'=u_{j_pi}\in V'$ for some $0\leq i\leq k$},\\
				0, &\textrm{otherwise};
			\end{array}\right. ,
		\mbox{ for any } p\in\delta(u',\tau)
		\}, 
\end{aligned}
\]
for any $u'\in V$ and $\tau\in\Sigma'$.
By the construction, $T_{\mathcal{S}}(u,k)$ is a tree. Now we illustrate the construction of $T_{\mathcal{S}}(u,k)$ by the example below.
	\begin{Eg}
		{\rm Let us reconsider the NFTS $\mathcal{S}=(V,\Sigma,\delta)$ shown in Fig. \ref{Figure2}. Then
			\begin{align*}
				D_u[0]=\{u\}, \quad D_u[1]=\{u_1,u_2\}, \quad
				D_u[3]=\{u_3,u_4\}, \quad &D_u[4]=\{u_2\}.
			\end{align*}
			If we let $u_0=u$, the 3-neighbouring path induced subsystem with respect to $u$ is $T_{\mathcal{S}}(u,k)=(V',\Sigma',\delta')$, where
			\begin{align*}
				&S'=\{u_{00},u_{11},u_{21},u_{32},u_{42},u_{23}\}, \quad \Sigma'=\{\tau_1,\tau_2,\tau_3\},\\
				&\delta'(u_{00},\tau_1)=\{\frac{0.2}{u_{11}}+\frac{0.7}{u_{21}}\}, \quad \delta'(u_{11},\tau_2)=\{\frac{0.3}{u_{32}}\}, \quad \delta'(u_{21},\tau_2)=\{\frac{0.4}{u_{42}}\}, \quad \delta'(u_{42},\tau_3)=\{\frac{0.4}{u_{23}}\}.
			\end{align*}
		}
	\end{Eg}

	Then by the construction of $k$-neighbouring path induced subsystem, Definition \ref{Def k-limited bisimulation}, and Lemma \ref{Prop lifted relation}, the following result holds.
	
	\begin{Lem}\label{Lem s sim t}
		{Let $\mathcal{S}=(V,\Sigma,\delta)$ be an NFTS, $k\in\mathbb{N}$ and $\alpha\in[0,1]$. For any $u,v\in V$, $u'\in(\overleftarrow{R_{\delta}})_u$ and $v'\in(\overleftarrow{R_{\delta}})_v$, $u\sim^{\alpha}v$ in $T_{\mathcal{S}}(u,k)$ and $T_{\mathcal{S}}(v,k)$ iff $u\sim^{\alpha}v$ in $T_{\mathcal{S}}(u',k+1)$ and $T_{\mathcal{S}}(u',k+1)$.
		}
	\end{Lem}
	
	Based on $\alpha$-bisimilarity under $\widetilde{S}$, an equivalent condition for determining two states to be $k$-limited $\alpha$-bisimilar is presented below.
	\begin{Thm}\label{Thm s approx t iff sim}
		{Given an NFTS ${\mathcal{S}}=(V,\Sigma,\delta)$, $k\in\mathbb{N}$, $\alpha\in[0,1]$ and $u,v\in V$. Then $u\approx_k^{\alpha}v$ iff $u\sim^{\alpha}v$ in $T_{\mathcal{S}}(u,k)$ and $T_{\mathcal{S}}(v,k)$.
		}
	\end{Thm}
	\begin{proof}
		(Sufficiency.) The sufficiency holds for $k=0$. With the sufficiency holds for $k=m$, $m\geq0$, and $u\sim^{\alpha}v$ in $T_{\mathcal{S}}(u,k+1)$ and $T_{\mathcal{S}}(v,k+1)$, we show that $u\approx_{k+1}^{\alpha}v$. Let $R$ be the greatest $\alpha$-bisimulation under $\widetilde{S}$ between $T_{\mathcal{S}}(u,k+1)$ and $T_{\mathcal{S}}(v,k+1)$. Since $(u,v)\in R$, $u\os{\tau}\lra p$ implies $v\os{\tau}\lra q$ such that $(p,q)\in R_{\alpha}^{\dag}$. By Definition \ref{relation lifting}, we only need to prove that $({\rm su}(p)\times {\rm su}(q))\cap R$ is an $m$-limited $\alpha$-bisimulation.
		For any $(u',v')\in({\rm su}(p)\times {\rm su}(q))\cap R$, $u'\sim^{\alpha}v'$ in $T_{\mathcal{S}}(u,{m+1})$ and $T_{\mathcal{S}}(v,m+1)$. Since $u\in(\overleftarrow{R_{\delta}})_{u'}$ and $v\in(\overleftarrow{R_{\delta}})_{v'}$, by Lemma \ref{Lem s sim t}, $u'\sim^{\alpha}v'$ in $T_{\mathcal{S}}(u',m)$ and $T_{\mathcal{S}}(v',m)$. By the induction hypothesis, $u'\approx_m^{\alpha}v'$, and so $({\rm su}(u)\times {\rm su}(q))\cap R$ is an $m$-limited $\alpha$-bisimulation. By Lemma \ref{Prop lifted relation} (2), $(p,q)\in(\approx_{m}^{\alpha})_{\alpha}^{\dag}$. The symmetric case that any transition from $t$ can be analyzed analogously. Hence $u\approx_{m+1}^{\alpha}v$ by Definition \ref{Def k-limited bisimulation}. The sufficiency holds.
		
		(Necessity.) Since $T_{\mathcal{S}}(u,0)=(\{u\},\emptyset,\emptyset)$ and $T_{\mathcal{S}}(v,0)=(\{v\},\emptyset,\emptyset)$, $u\sim^{\alpha}v$ in $T_{\mathcal{S}}(u,0)$ and $T_{\mathcal{S}}(v,0)$, i.e., the necessity holds for $k=0$. With the condition that the necessity holds for $k\geq 0$, we testify that the necessity holds for $k+1$.
		By $u\approx_{k+1}^{\alpha}v$, $u\os{\tau}\lra p$ implies $v\os{\tau}\lra q$ that satisfies  $(p,q)\in(\approx_k^{\alpha})_{\alpha}^{\dag}$. Subsequently, we derive that $(p,q)\in(({\rm su}(p)\times {\rm su}(q))\cap\approx_k^{\alpha})_{\alpha}^{\dag}$ from Definition \ref{relation lifting}, and so showing  that $({\rm su}(p)\times {\rm su}(q))\cap\approx_k^{\alpha}$ is an $\alpha$-bisimulation under $\widetilde{S}$ between $T_{\mathcal{S}}(u,k)$ and $T_{\mathcal{S}}(v,k)$ is enough.
		For any $(u',v')\in({\rm su}(p)\times {\rm su}(q))\cap\approx_k^{\alpha}$, since $u\in(\overleftarrow{R_{\delta}})_{u'}$,
		$v\in(\overleftarrow{R_{\delta}})_{v'}$, and $u'\sim^{\alpha}v'$ in $T_{\mathcal{S}}(u,k+1)$ and $T_{\mathcal{S}}(v,k+1)$, by Lemma \ref{Lem s sim t}, $u'\sim^{\alpha}v'$ in $T_{\mathcal{S}}(u',k)$ and $T_{\mathcal{S}}(v',k)$, and so $u'\approx_k^{\alpha}v'$. It is known that the necessity holds in case $k$, hence $u'\sim^{\alpha}v'$ in $T_{\mathcal{S}}(u',k)$ and $T_{\mathcal{S}}(v',k)$, which implies that $({\rm su}(u)\times {\rm su}(v))\cap\approx_k^{\alpha}$ is an $\alpha$-bisimulation under $\widetilde{S}$. Let $R_1$ be the $\alpha$-bisimilarity under $\widetilde{S}$ between $T_{\mathcal{S}}(v,k+1)$ and $T_{\mathcal{S}}(v,k+1)$. Then ${\rm su}(p)\times {\rm su}(q))\cap\approx_k^{\alpha}\subseteq R_1$, and hence $(p,q)\in(R_1)_{\alpha}^{\dag}$ by Lemma \ref{Prop lifted relation} (2).
		Without changing the analysis method, the other assertion can be obtained. By Definition \ref{Def app-bisimulation} and $R_1$ is the greatest $\alpha$-bisimulation under $\widetilde{S}$ between $T_{\mathcal{S}}(u,k+1)$ and $T_{\mathcal{S}}(v,k+1)$, we have $(u,v)\in R_1$, i.e., $u\sim^{\alpha}v$ in $T_{\mathcal{S}}(u,k+1)$ and $T_{\mathcal{S}}(v,k+1)$. We complete the induction steps.
	\end{proof}
	
	For the special case that $k$ is the lager number of maximum lengths $l(u)$ from $u$ and $l(v)$ from $v$, the following fact holds.
	\begin{Coro}
		{Let $\mathcal{S}=(V,\Sigma,\delta)$ be an NFTS, $k\in\mathbb{N}$, $\alpha\in[0,1]$, and $u,v\in V$. Then $u\approx_{l(u)\vee l(v)}^{\alpha}v$ iff $u\sim^{\alpha}v$.
		}
	\end{Coro}
	\begin{proof}
		We only prove that the necessity is valid. It is easy to obtain that the necessity holds for $l(u)\vee l(v)=0$. With the assume that the statement holds for $l(u)\vee l(v)=k$, and $u\approx_{l(u)\vee l(v)}^{\alpha}v$ for $l(u)\vee l(v)=k+1$. By Theorem \ref{Thm s approx t iff sim}, we have that $u\sim^{\alpha}v$ in $T_{\mathcal{S}}(u,k+1)$ and $T_{\mathcal{S}}(v,k+1)$, and hence $u\sim^{\alpha}v$.
	\end{proof}
	
	\subsection{Logical characterization of $k$-limited $\alpha$-bisimulation}
	This subsection discusses the Hennessy-Milner property of $k$-limited $\alpha$-bisimulations in terms of the equivalence between the limited bisimulation and $\alpha$-bisimulation under $\widetilde{S}$, along with the fuzzy modal logic proposed by Qiao and Zhu \cite{Qiao(2022)} and adopted from \cite{Wuandchen(2018)}, which is given by  
	\begin{align*}
		\varphi\ \colon\!\!\!\colon\!\!\!\!=&\mathrm{T}\mid\varphi_1\wedge\varphi_2\mid \varphi\rra s\mid s\rra\varphi\mid\varphi\otimes s\mid\langle \tau\rangle\psi, \\ 
		\psi\ \colon\!\!\!\colon\!\!\!\!=&\psi_1\wedge\psi_2\mid\psi\rra s\mid s\rra\psi\mid{\varphi}^{\dag},
	\end{align*}
	where $\tau\in A$, $s\in[0,1]$, $\varphi\in\mathcal{L}^u$ is a state formula and $\psi\in\mathcal{L}^d$ is a distribution formula which are interpreted in an NFTS as follows \cite{Qiao(2022),Wuandchen(2018)}. 
	\begin{align*}
		&\llbracket{\mathrm{T}}\rrbracket(u)=1,\\ &\llbracket{\varphi_1\wedge\varphi_2}\rrbracket(u)=\llbracket{\varphi_1}\rrbracket(u)\wedge\llbracket{\varphi_2}\rrbracket(u),\\
		&\llbracket{\varphi\rra s}\rrbracket(u)=\llbracket{\varphi}(u)\rrbracket\rra s,\\
		&\llbracket{s\rra\varphi}\rrbracket(u)=s\rra\llbracket\varphi\rrbracket(u),\\
		&\llbracket{\varphi\otimes s}\rrbracket(u)=(\llbracket\varphi\rrbracket(u)\otimes s),\\
		&\llbracket{\la \tau\ra\psi}\rrbracket(u)=\bigvee_{u\os{\tau}\lra p}\llbracket{\psi}\rrbracket(p),
	\end{align*}
	\begin{align*}
		&\llbracket{\psi_1\wedge\psi_2}\rrbracket(p)=\llbracket{\psi_1}\rrbracket(p)\wedge\llbracket{\psi_2}\rrbracket(p),\\
		&\llbracket{\psi\rra s}\rrbracket(p)=\llbracket{\psi}\rrbracket(p)\rra s,\\
		&\llbracket{s\rra\psi}\rrbracket(p)=s\rra\llbracket{\psi}\rrbracket(p),\\
		&\llbracket{{\varphi}^{\dag}}\rrbracket(p)=\bigvee_{u\in V}p(u)
		\wedge\llbracket{\varphi}\rrbracket(u).
	\end{align*}
	
	As in \cite{Qiao(2022),Wu(2016),Wuandchen(2018)}, $\llbracket{\la \tau\ra\psi}\rrbracket(u)$ is defind as the maximal $\llbracket \psi\rrbracket(p)$ with the condition $u\os{\tau}\lra p$.
	\cite{Qiao(2022)} characterizes $\alpha$-bisimilarity under $\widetilde{S}$ resorting to the real-valued logic. Please refer to \cite{Qiao(2022),Wu(2016),Wuandchen(2018)} for explanations of other formulae. With the help of $k$-neighbouring path induced subsystems, real-valued logical characterization of $k$-limited $\alpha$-bisimilarity can be transformed into that of $\alpha$-bisimilarity under $\widetilde{S}$.
	\begin{Thm}\label{Logical k characterization}
		{ Let $\mathcal{S}=(V,\Sigma,\delta)$ be finitary, $\alpha\in[0,1]$, and $u,v\in V$.
			Suppose that the implication is induced by a G\"{o}del $t$-norm, or for any $r,s,t,o\in L$, the implication satisfies the following conditions:
			\begin{align*}
				& \mbox{If } s\otimes t>0, \mbox{ then } s\rra s\otimes t=t.                     \\
				& \mbox{If } t<r\leq s, \mbox{ then } s\rra t<s\rra r.                           \\
				& \mbox{If } t<r \mbox{ and } t\otimes o>0, \mbox{ then } t\otimes o<r\otimes o.
			\end{align*}
			Then $u\approx_k^{\alpha}v$ iff $\llbracket{\varphi}\rrbracket(u)\leftrightarrow\llbracket{\varphi}\rrbracket(v)\geq\alpha$ for any $\varphi\in\mathcal{L}^u$ in $T_{\mathcal{S}}(u,k)$ and $T_{\mathcal{S}}(v,k)$.
		}
	\end{Thm}
	\begin{proof}
		By Theorem \ref{Thm s approx t iff sim}, $u\approx_k^{\alpha}v$ iff $u\sim^{\alpha}v$ in $T_{\mathcal{S}}(u,k)$ and $T_{\mathcal{S}}(v,k)$. By Theorems 6.1 and 6.2 in \cite{Qiao(2022)}, $u\sim^{\alpha}v$ in $T_{\mathcal{S}}(u,k)$ and $T_{\mathcal{S}}(v,k)$ iff $\llbracket{\varphi}\rrbracket(u)\leftrightarrow\llbracket{\varphi}\rrbracket(v)\geq\alpha$ for any $\varphi\in\mathcal{L}^u$ in $T_{\mathcal{S}}(u,k)$ and $T_{\mathcal{S}}(v,k)$. Hence $u\approx_k^{\alpha}v$ iff $\llbracket{\varphi}\rrbracket(u)\leftrightarrow\llbracket{\varphi}\rrbracket(v)\geq\alpha$ for any $\varphi\in\mathcal{L}^u$ in $T_{\mathcal{S}}(u,k)$ and $T_{\mathcal{S}}(v,k)$. We complete the proof.
	\end{proof}
	\begin{Rem}\label{Thm k loic}
		{\rm The implications can be induced by at least {\L}ukasiewicz norm, Product norm, and G\"{o}del norm.
		}
	\end{Rem}

	\section{Conclusion and future work}
	We have investigated $k$-limited $\alpha$-bisimilarity for NFTSs in the neighbouring subgraphs. Using fixed point theory, the fixed-pint characterization of limited bisimulation has been completed via constructing a monotonic function of which a post-fixed point is a limited bisimulation. Then, through $k$-limited $\alpha$-bisimilarity, $k$-limited $\alpha$-bisimulation and $k$-limited $\alpha$-bisimulation vector have been introduced. The equivalence between them has been discussed. Moreover,
	conditions for two states to be $k$-limited $\alpha$-bisimilar have been discussed. These theoretic results build a foundation for the design of the algorithm for computing the degree of similarity. 
	Based on $k$-neighbouring path induced subsystem, two states are $k$-limited $\alpha$-bisimilar when and only when they are $\alpha$-bisimilar under lifting function $\widetilde{S}$ in their $k$-neighbouring path induced subsystems. 
	With this equivalence, a logical characterization of $k$-limited $\alpha$-bisimilarity has been showed and it has proved that two states $u$ and $v$ are $k$-limited $\alpha$-bisimilar iff $\llbracket{\varphi}\rrbracket(u)\leftrightarrow\llbracket{\varphi}\rrbracket(v)\geq\alpha$ for any real-valued formula $\varphi\in\mathcal{L}^u$ in their $k$-neighbouring path induced subsystems.
	
	Future work of interest may be on the state reduction of NFTSs and graph pattern matching resorting to limited bisimulations proposed, along with limited fuzzy bisimulations and distribution-based limited fuzzy bisimulations for NFTSs. It is interesting to introduce limited bisimulation for Boolean networks and Boolean control networks to study their quotients \cite{Li(2021),Li(2023),Zhao(2023)}.   
	
\section*{Acknowledgements}
This work was supported by the National Natural Science Foundation of China under Grants 12301589, 62273201, 62172048, and 62303170, and by the Research Fund for the Taishan Scholar Project of Shandong Province of China under Grant tstp20221103.



\end{document}